\documentclass{amsart}
\usepackage{amsmath,amssymb,amsthm,amscd,verbatim,graphicx,color}
\usepackage[mathscr]{eucal}
\usepackage{enumerate,color}
\usepackage{pgf}
\usepackage{tikz}
\usetikzlibrary{decorations.pathreplacing}%,hobby}
\theoremstyle{plain}
\newtheorem{thm}{Theorem}[section]
\newtheorem{lem}[thm]{Lemma}
\newtheorem{cor}[thm]{Corollary}

\newtheorem{example}[thm]{Example}
\newtheorem{eg}[thm]{Example}

\theoremstyle{definition}
\newtheorem{defn}[thm]{Definition}

\newtheorem{remark}[thm]{Remark}
\newtheorem{notation}[thm]{Notation}

\makeatletter

\newcommand{\Rmnum}[1]{\expandafter\@slowromancap\romannumeral #1@}
\makeatother

\newcommand{\cat}[1]{\boldsymbol{\mathscr{#1}}}

\DeclareSymbolFont{bbold}{U}{bbold}{m}{n}
\DeclareSymbolFontAlphabet{\mathbbold}{bbold}
\DeclareMathOperator{\CSP}{CSP}
\DeclareMathOperator{\Inc}{Inc}
\DeclareMathOperator{\Block}{Block}
\DeclareMathOperator{\dist}{dist}
\newcommand{\up}[1]{\textup{#1}}

\renewcommand{\ge}{\geqslant}

\newcommand{\Mod}{\operatorname{Mod}}
\newcommand{\Th}{\operatorname{Th}}
\renewcommand{\emptyset}{\varnothing}

\newcommand{\lrge}{\operatorname{large}}

\begin{document}

\title{Relativised homomorphism preservation at the finite level}

\author{Lucy Ham}
\address{Department of Mathematics and Statistics\\ La Trobe University\\ Victoria  3086\\
Australia} \email{leham@students.latrobe.edu.au}
\thanks{The author would like to thank her PhD supervisors Brian Davey and Tomasz Kowalski, as well as Marcel Jackson for many useful discussions}

%\subjclass[2000]{Primary: ?????; Secondary: ?????}

\begin{abstract}
In this article, we investigate the status of the homomorphism preservation property amongst restricted classes of finite relational structures and algebraic structures. We show that there are many homomorphism-closed classes of finite lattices that are definable by a first-order sentence but not by existential positive sentences, demonstrating the failure of the homomorphism preservation property for lattices at the finite level. In contrast to the negative results for algebras, we establish a finite-level relativised homomorphism preservation theorem in the relational case. More specifically, we give a complete finite-level characterisation of first-order definable finitely generated anti-varieties relative to classes of relational structures definable by sentences of some general forms. When relativisation is dropped, this gives a fresh proof of Atserias's characterisation of first-order definable constraint satisfaction problems over a fixed template, a well known special case of Rossman's Finite Homomorphism Preservation Theorem.
\end{abstract}
%give a new proof of Atserias's characterisation of first-order definable constraint satisfaction problems, a well known special case of the Homomorphism Preservation Theorem for finite relational structures. We generalise this result slightly to give an analogous finite level characterisation of first-order definable finitely generated anti-varieties. Moreover, we show that these results relativise to any class of finite relational structures closed under taking homomorphisms, products and disjoint unions. We contrast these positive results for relational structures with some examples of classes of finite algebras that exhibit the failure of the Homomorphism Preservation Theorem at the finite level. More specifically, we show that there are many homomorphism closed classes of finite lattices that are definable by a first-order sentence but not by existential positive sentences.

\maketitle

%%%%%%%%%%%%%%%%%%%%%%%%%%%%%%%%%%%%%%%%%%%%%%%%%%%%%%%%%%%%%%%%%%
%%%%%%%%%%%%%%%%%%%%%%%%%%%%%%%%%%%%%%%%%%%%%%%%%%%%%%%%%%%%%%%%%%
\section{Introduction}
\subsection*{Finite-level preservation theorems: algebras vs relations}
It is widely known that many classical results in model theory fail or become meaningless when considered at the finite level. An important class of such results are preservation theorems, which relate the form of syntactic expressions to the closure properties of classes they define. Classical preservation theorems are almost uniformly false when restricted to the class of all finite relational structures, but also tend to fail when restricted to the class all of finite algebras. A key example is the {\L}os-Tarski Preservation Theorem, which says that a first-order sentence is preserved under taking extensions if and only if it is equivalent to an existential sentence. See Tait~\cite{Tait} or Gurevich and Shelah~\cite{Gurevich} for a counterexample in the finite setting for relational structures, and Clark, Davey, Jackson and Pitkethly~\cite[Example $4.3$]{CDJP} for a counterexample in the finite setting for algebras. For some preservation theorems only one side of the story is explained; for example the ${\mathbb I}{\mathbb S}{\mathbb P}$-Preservation Theorem, which asserts that a first-order formula is preserved under taking substructures and direct products if and only if it is equivalent to a universal-Horn sentence, remains to be settled at the finite-level in the relational setting. In contrast, Example 4.3 of \cite{CDJP} also gives a counterexample for algebraic structures at the finite level. On the other hand, Lyndon's Positivity Theorem, which says that a first-order sentence is preserved under taking surjective homomorphisms if and only if it is equivalent to a positive sentence, fails in the finite setting for relational structures (see Ajtai and Gurevich~\cite{Ajtai}), while its status remains unknown in the algebraic setting.  A notable exception is Rossman's celebrated Finite Homomorphism Preservation Theorem~\cite[Theorem 5.16]{Rossman}: a first-order sentence is preserved under homomorphisms on finite relational structures if and only if it is equivalent, amongst finite relational structures, to an existential positive sentence. The question whether the classical Homomorphism Preservation Theorem would continue to hold on finite relational structures was a long-standing one in finite model theory. The analogous question for algebraic signatures remains unresolved. The proof for the relational case is highly non-trivial and, as Rossman suggests, the task of adjusting for function symbols is conceivably more difficult.
\subsection*{Relativised preservation theorems}
Preservation theorems that hold for the class of all finite relational structures do not necessarily relativise to subclasses of this class, and moreover, while many classical preservation theorems fail amongst all finite relational structures, some preservation theorems are recovered in restricted classes of finite relational structures. In \cite{ADG}, Atserias, Dawar, and Grohe show that the {\L}os-Tarski Preservation Theorem is recovered in a number of classes of finite relational structures that are of interest from a computational point of view, including classes of finite acyclic graphs, classes of structures of bounded degree, and classes of structures of treewidth at most $k$, for any positive integer $k$. In \cite{ADK}, Atserias, Dawar and Kolaitis show that the homomorphism preservation property holds for a number of classes of finite relational structures, including classes of structures of bounded degree, classes of bounded treewidth, and more generally, classes of finite structures whose cores exclude at least one minor. However these results are not implied by Rossman's Finite Homomorphism Preservation Theorem. Furthermore, preservation theorems need not behave consistently when we restrict our attention to a particular subclass of the class of all finite relational structures: for example, the homomorphism preservation property holds for the class of all planar graphs, but in contrast, the extension preservation property fails in this setting.

\subsection*{Main results}
 Relativising preservation theorems to restricted classes of finite structures appears to be completely independent of whether or not the corresponding preservation property holds amongst all finite relational structures. On the other hand, relativisation of preservation theorems to restricted classes of finite algebras is not as well studied. In this paper, we are able to give some insight into this topic.
\begin{enumerate}
\item \label{item:1} We show that within the class of finite bounded lattices, there are many classes that demonstrate the failure of the homomorphism preservation property. More specifically, we show that for a fixed finite bounded lattice ${\mathbf L}$, the class of finite bounded lattices admitting a homomorphism into ${\mathbf L}$ is definable by a single first-order sentence, relative to the class of finite bounded lattices, but is not definable by universal sentences. See Theorem~\ref{thm:algfail} in Section~\ref{section:SectionAlg}.

\item \label{item:2} In contrast to (\ref{item:1}), we establish a relativised homomorphism preservation theorem (see Theorem~\ref{thm:Sausage} in Section~\ref{section:SectRelational}) for a wide range of classes of finite relational structures. Theorem~\ref{thm:Sausage} gives us a complete characterisation of first-order definable classes of relational structures admitting a homomorphism into a fixed finite relational structure, relative to classes ${\cat K}$ definable by sentences of some general forms. In Theorem~\ref{thm:colourfam} we extend this characterisation to classes of relational structures admitting a homomorphism into one of a finite number of relational structures.
\end{enumerate}
 We will see that Theorem~\ref{thm:algfail} mentioned in (\ref{item:1}) not only demonstrates the failure of the homomorphism preservation property for the class of finite bounded lattices, but also the failure of the extension preservation property ({\L}os-Tarski Preservation Theorem). So the extension preservation property, which fails for the class of all finite algebras and also the class of all finite relational structures, continues to fail when we restrict to the class of finite bounded lattices. This result could be contrasted with the positive results of Atserias, Dawar and Grohe~\cite{ADG} mentioned previously, which show that the extension preservation property is recovered in a number of restricted classes of finite relational structures.

 When the class ${\cat K}$ in Theorem~\ref{thm:Sausage} is relaxed to include all finite relational structures, our result coincides with a well known one due to Atserias, which says that first-order definable constraint satisfaction problems are precisely those with finite tree duality (or equivalently, those definable by existential positive sentences), see Theorem $9$ of \cite{AA}. It is also a consequence Rossman's Finite Homomorphism Preservation Theorem, which holds for classes of relational structures admitting a homomorphism into one of \emph{any} number of finite relational structures.  Our proof, however, is independent of both Atserias and Rossman and relies mostly on some constructions borrowed from Larose, Loten and Tardif given in~\cite{LLT}. We remark that while it remains unclear if the homomorphism preservation property in its full generality holds in these classes, the restricted preservation property, nevertheless, covers many commonly encountered homomorphism-closed classes.
%give an entirely new proof that a class of finite relational structures admitting a homomorphism into a fixed relational structure is first-order definable if and only if it is definable by anti-identities (negations of existential positive sentences). This result coincides with Atserias's characterisation of first-order definable constraint satisfaction problems are precisely those with finite tree duality and is a consequence of Rossman's Finite Homomorphism Preservation Theorem. However, it should be remarked that Atserias established this result almost simultaneously and independently of Rossman's Finite Homomorphism Preservation Theorem, and furthermore,. From our new proof, it is apparent the classes on which this result will relativise. As it turns out, these classes are of general form and cover an extensive range of classes of finite relational structures. We should point out that the relativised version of this result has not previously been established nor is it a consequence of Atserias's proof or Rossman's. It is also worth pointing out that it remains unclear if the classical homomorphism preservation property holds for these classes.

%Does not tell us anything about the potential failure of the Finite Homomorphism Preservation Theorem for algebras since we cannot conclude anything about the status of a preservation theorem on a certain class from its status on a subclass, and vice-versa.

\subsection*{Article structure}
This article is comprised of three main sections. We begin with some background and definitions that will be required for the remainder of the paper, we then investigate the status of Relativised Homomorphism Preservation Theorem~\ref{thm:relhompresthm} at the finite level for both algebraic and relational signatures.
In Section~\ref{section:SectionAlg}, we give counterexamples to the finite-level version of Relativised Homomorphism Preservation Theorem~\ref{thm:relhompresthm} for algebraic structures. Moreover, we completely characterise those classes of finite lattices admitting a homomorphism into a fixed finite bounded lattice.
In Section~\ref{section:SectRelational}, we establish a restricted form of the finite-level version of Relativised Homomorphism Theorem~\ref{thm:relhompresthm} for relational structures. The proof follows an Ehrenfeucht-Fra\"{\i}ss\'e game argument on structures obtained by adjusting some constructions borrowed from Larose, Loten and Tardif given in~\cite{LLT}.

\section{Preliminaries}

\begin{defn}
A \emph{signature} (or \emph{language} or \emph{type}) is a pair ${\mathcal L}=\langle {\mathcal R}, {\mathcal F}\rangle$ consisting of a set of relation symbols ${\mathcal R}$ and a set of function symbols ${\mathcal F}$. Each relation and function symbol has an associated \emph{arity} $n$; we do not allow for nullary relation symbols so $n$ is quantified over ${\mathbb N}$ in this case; we do however include the possibility of $0$-ary function symbols so $n$ is quantified over ${\mathbb N}\cup\{0\}$ in this case.
An ${\mathcal L}$-\emph{structure}  ${\bf A}=\langle A; {\mathcal R}^{\bf A}, {\mathcal F}^{\bf A}\rangle$ consists of a universe $A$ together with an interpretation of
\begin{itemize}
\item each $n$-ary relation symbol $R\in {\mathcal R}$ as a subset $R^{\bf A}$ of ${A}^n$,
\item each $n$-ary function symbol $F\in {\mathcal F}$ as a function $F^{\bf A}\colon A^{n}\to A$.
\end{itemize}
If ${\mathcal R}=\emptyset$, then ${\bf A}$ is an \emph{algebra}, if ${\mathcal F}=\emptyset$ then ${\bf A}$ is a \emph{relational structure}.
\end{defn}
\begin{defn}
Let ${\mathbf A}$ be a structure in some signature ${\mathcal L}$ consisting of possibly both relation and function symbols. We let ${\mathcal L}_{\mathbf A}$ be the language obtained by adding a nullary function symbol $a$ to ${\mathcal L}$, for each $a\in A$.
\end{defn}
\begin{defn}
Let $n\in{\mathbb N}$, let ${\mathcal R}=\{R_1, \dots, R_n\}$ be a purely relational signature, and let $\mathbf A = \langle A; R_1^{\mathbf A},\dots,R_n^{\mathbf A}\rangle$. A \emph{weak substructure} of $\mathbf A$ is a structure $\mathbf A'=\langle A';R_1^{{\mathbf A}'},\dots,R_n^{{\mathbf A}'}\rangle$, where $A'\subseteq A$ and $R_i^{{\mathbf A}'}\subseteq R_i^{\mathbf A}$ for each $i\in\{1,\dots,n\}$.  The weak substructure $\mathbf A'$ is \emph{proper} if either $A'\subsetneqq A$ or $A'=A$ and there exists $i\leq n$ such that $R_i^{{\mathbf A}'}\subsetneqq R_i^{\mathbf A}$.
\end{defn}
\begin{defn}
Let ${\mathcal L}$ be some signature consisting of possibly both relation and function symbols, and let ${\cat J}$ be a class of ${\mathcal L}$-structures. We define ${\mathbb H}^{\gets}({\cat J})$ to be the class of all ${\mathcal L}$-structures admitting a homomorphism into some member of the class~${\cat J}$.
\end{defn}

\begin{defn}
Let $I$ be a set and let ${\cat F}=\{{\mathbf A}_i \mid i\in I\} $ be a set of structures in some signature ${\mathcal L}$. We say that $\cat F$ is a \emph{homomorphism-independent} set if for all $i, j\in I$ with $i\ne j$, there is no homomorphism from ${\mathbf A}_i$ to ${\mathbf A}_j$.
\end{defn}

%\begin{defn}
%Let ${\mathbf A}$ be a finite structure in the purely relational signature ${\mathfrak R}=\{R_1, \dots, R_m\}$. The class of all \emph{finite} ${\mathfrak R}$-structures admitting a homomorphism into ${\mathbf A}$ is denoted $\CSP({\mathbf A})$. The membership problem for $\CSP({\mathbf A})$ is known as the \textbf{constraint satisfaction problem over} template ${\mathbf A}$. An \textbf{obstruction} of the class $\CSP({\mathbf A})$ is any ${\mathfrak R}$-structure not belonging to $\CSP({\mathbf A})$.
%An obstruction ${\mathbf C}$ of the class $\CSP({\mathbf A})$ is \textbf{critical} if every proper weak substructure of ${\mathbf C}$ belongs to $\CSP({\mathbf A})$.
%\end{defn}

%\begin{defn}
%Let ${\mathbf A}$ be a finite structure in the purely relational signature ${\mathfrak R}=\{R_1, \dots, R_m\}$.
%The class $\CSP({\mathbf A})$ satisfies the \textbf{finite duality property} if there exists a finite set ${\cat S}$ of finite ${\mathfrak R}$-structures with the following property\textup: an ${\mathfrak R}$-structure ${\mathbf B}$ belongs to the set $\CSP({\mathbf A})$ if and only if ${\mathbf B}$ avoids a homomorphism from every member of ${\cat S}$.
%\end{defn}

\begin{defn}
Let $n,m\in {\mathbb N}$ and let ${\cat F}=\{{\mathbf A}_1, \dots {\mathbf A}_n\}$ be a finite set of finite structures in the purely relational signature ${\mathcal R}=\{R_1, \dots, R_m\}$. We denote the finite models in ${\mathbb H}^{\gets}({\cat F})$ by ${\mathbb H}^{\gets}_{\textup{fin}}({\cat F})$. An \emph{obstruction} for the class ${\cat F}$ is any finite ${\mathcal R}$-structure not belonging to ${\mathbb H}^{\gets}_{\textup{fin}}({\cat F})$. An obstruction ${\mathbf C}$ for the class ${\cat F}$ is \emph{critical} if every proper weak substructure of ${\mathbf C}$ belongs to ${\mathbb H}^{\gets}_{\textup{fin}}({\cat F})$.
We say that ${\mathbb H}^{\gets}_{\textup{fin}}({\cat F})$ satisfies the \emph{finite duality property} if there exists a finite set ${\cat S}$ of finite structures with the following property\textup: a finite ${\mathcal R}$-structure ${\mathbf B}$ belongs to the class ${\mathbb H}^{\gets}_{\textup{fin}}({\cat F})$ if and only if ${\mathbf B}$ avoids a homomorphism from every member of ${\cat S}$.
Let ${\mathbf A}$ be a finite structure in the purely relational signature ${\mathcal R}$.

The class of all finite ${\mathcal R}$-structures admitting a homomorphism into ${\mathbf A}$ is denoted $\CSP({\mathbf A})$. The membership problem for $\CSP({\mathbf A})$ is known as the \emph{constraint satisfaction problem} over template ${\mathbf A}$.
Note that ${\mathbb H}^{\gets}_{\textup{fin}}({\cat F})=\bigcup\limits_{i\in\{1, \dots, n\} }\CSP({\mathbf A}_i)$.
\end{defn}

\begin{defn}
Let $m\in{\mathbb N}$, let ${\mathbf A}$ be a finite relational structure in the purely relational signature ${\mathcal R}=\{R_1, \dots, R_m\}$ with $R_i$ of arity $r_i$, for all $i\in\{1,\dots, m\}$. The \emph{incidence multigraph} of ${\mathbf A}$, denoted $\Inc({\mathbf A})$, is defined as the bipartite multigraph with parts $A$ and $\Block({\mathbf A})$, where $\Block({\mathbf A})$ consists of all pairs $(R_i, (a_1, \dots, a_{r_i}))$ such that $R_i\in\{R_1, \dots R_m\}$ and $(a_1, \dots, a_{r_i})\in R_i^{\mathbf A}$, and with edges $e_{a,k,B}$ joining $a\in A$ to $B=(R_j, (a_1, \dots, a_{r_j}))\in \Block({\mathbf A})$, when $a_k=a$ for some $k\in\{1, \dots, r_j\}$. Let $a$ and $b$ be two elements in $A$. Then the \emph{distance between $a$ and $b$ in} ${\mathbf A}$ $\dist_{\mathbf A}(a,b)$ is defined to be half the number of edges in a shortest path connecting them in $\Inc({\mathbf A})$. The \emph{girth} of ${\mathbf A}$ is defined to be half the shortest length of a cycle in $\Inc({\mathbf A})$. The \emph{diameter} of ${\mathbf A}$ is defined to be $\max\{\dist_{\mathbf A}(u,v)\mid (u,v)\in {A^2}\}$.
\end{defn}

\begin{defn}
Let ${\cat K}$ be a class of ${\mathcal L}$-structures and let ${\cat J}$ be a subclass of ${\cat K}$. We define the following class: \[
{\mathbb H}_{\cat K}^{\gets}({\cat J}):=\{{\mathbf B}\in {\cat K}\ |\ \text{there exists}\ {\mathbf J}\in {\cat J}\ \text{and a homomorphism from}\ {\mathbf B}\ \text{to}\ {\mathbf J}\}.
\]
In the case where ${\cat K}$ is the class of all ${\mathcal L}$-structures we abbreviate ${\mathbb H}_{\cat K}^{\gets}$ to ${\mathbb H}^{\gets}$.  In general, ${\mathbb H}_{\cat K}^{\gets}({\cat J})={\mathbb H}^{\gets}({\cat J})\cap {\cat K}$.
\end{defn}
Let $n\in \mathbb{N}$, let ${\mathbf A}\in {\cat K}$, and let ${\cat F}=\{{\mathbf A_1}, \dots, {\mathbf A_n}\}$ be a finite set of structures in ${\cat K}$. We follow Gorbunov~\cite{Gorbachev} and refer to the class ${\mathbb H}_{\cat K}^{\gets}({\mathbf A})$ as a \emph{principal colour-family in ${\cat K}$}, and we refer to the class ${\mathbb H}_{\cat K}^{\gets}({\cat F})$ as a \emph{colour-family in} ${\cat K}$.
\begin{defn}
Let $n\in {\mathbb N}\cup\{0\}$ and let $m\in{\mathbb N}$.  A sentence $\psi$ is called an \emph{anti-identity} if it is of the form
\[
\forall x_1, \dots, x_{n}\  [\neg \alpha_1(x_1, \dots, x_{n}) \vee \dots \vee \neg \alpha_{m}(x_1, \dots, x_{n})],
\]
where $\alpha_i(x_1,\dots, x_{n})$ is an atomic formula for all $i\in \{1, \dots, m\}$.
\end{defn}
\begin{notation}
Let $\Sigma$ be a set of first-order sentences in some signature $\mathcal{L}$. We let $\Mod({\Sigma})$ denote the class of all models of $\Sigma$.
\end{notation}.

\begin{defn}
A subclass $\cat K'$ of a class $\cat K$ is called \emph{strictly elementary relative to} ${\cat K}$ if there is a first-order sentence $\phi$ such that ${\cat K'}={\cat K} \cap \Mod(\phi)$.
\end{defn}

\begin{defn}
A subclass $\cat K'$ of a class $\cat K$ is called an \emph{anti-variety relative to} ${\cat K}$ if $\cat K'= {\cat K} \cap \Mod({\Sigma})$ where ${\Sigma}$ is a set of anti-identities.
\end{defn}

The proof of the Relativised Homomorphism Preservation Theorem~\ref{thm:relhompresthm} requires only standard applications of some model-theoretic fundamentals. We have included its proof, however, for completeness sake.
We remark that the results of this paper are more easily formulated for classes that are complements of homomorphism closed classes. Such classes are closed under taking homomorphic pre-images: classes ${\cat J}$ satisfying ${\mathbf B}\in{\cat J}$ whenever ${\mathbf B}$ admits a homomorphism to ${\mathbf A}$ and ${\mathbf A}\in{\cat J}$. For this reason, we work with the homomorphism preservation theorem in its complemented form.
\begin{thm}[Relativised Homomorphism Preservation Theorem]\label{thm:relhompresthm}
Let $\Sigma$ be a set of sentences in the signature ${\mathcal L}$, let ${\cat K}=\Mod(\Sigma)$, and let ${\cat J}$ be a subclass of~${\cat K}$.
Then the class ${\mathbb H}_{\cat K}^{\gets}({\cat J})$ is an anti-variety relative to ${\cat K}$. Moreover, if ${\mathbb H}_{\cat K}^{\gets}({\cat J})$ is strictly elementary with respect to ${\cat K}$, then there is a sentence consisting of a finite conjunction anti-identities that is logically equivalent to ${\phi}$ modulo $\Sigma$.
\end{thm}
\begin{proof}
 We first show that ${\mathbb H}_{\cat K}^{\gets}({\cat J})$ is axiomatisable by a set of anti-identities with respect to the class ${\cat K}$. The second part of the theorem will follow by compactness.

Let ${\cat J}^*={\mathbb H}_{\cat K}^{\gets}({\cat J})$ and let $\Th_{\text{anti}}({\cat J})$ be the set of anti-identities true in ${\cat J}$. Clearly, if ${\mathbf A}\in {\mathbb H}_{\cat K}^{\gets}({\cat J})$, then ${\mathbf A}\models\Th_{\text{anti}}({\cat J})$, since anti-identities are preserved under taking homomorphic pre-images.
Now let ${\mathbf A}\in {\cat K}$ and suppose that ${\mathbf A}\models\Th_{\text{anti}}({\cat J})$. Let ${\mathbf T}_{\mathbf A}$ be the set of atomic sentences true in ${\mathbf A}$ in the expanded signature ${\mathcal L}_{\mathbf A}$, and let ${\cat J}^*_{\mathbf A}$ be the class of all ${\mathcal L}_{\mathbf A}$-structures whose ${\mathcal L}$-reduct belongs to ${\cat J}^*$.
For any finite subset $\{\phi_1(a_1, \dots, a_{n}), \dots, \phi_{k}(a_1,\dots, a_{n})\}$ of ${\mathbf T}_{\mathbf A}$, we have
\[
{\mathbf A}\models \phi_1(a_1,\dots, a_{n})\wedge \dots \wedge \phi_{k}(a_1,\dots, a_{n}).
\]
So, ${\mathbf A}$ fails the sentence
\[
\sigma \equiv \forall x_1,\dots, x_{n}\ [\neg\phi_1(x_1, \dots, x_{n})\vee \dots \vee \neg\phi_{k}(x_1,\dots, x_{n})].
\]
Hence there is some structure ${\mathbf B}\in {\cat J}^*$ that fails $\sigma$.

Let $b_1, \dots, b_{n}$ be the elements of $B$ that witness the failure of the sentence $\sigma$. Now expand ${\mathbf B}$ in the signature ${\mathcal L}_{\mathbf A}$ under the following interpretation: let $b$ be any element of $B$, and define $a_i^{\mathbf B}:=b_i$, if $i\in\{1, \dots, {n}\}$, and $a_i^{\mathbf B}:=b$, otherwise. Then ${\mathbf B}_{\mathbf A}\in {\cat J}^*_{\mathbf A}$, and satisfies
\[
\phi_1(a_1,\dots, a_{n})\wedge \dots \wedge \phi_{k}(a_1,\dots, a_{n}).
\]
%{\mathbf B}\models \exists x_1, x_2 \dots, x_n [\phi_1(x_1,x_2,\dots, x_n)\wedge \dots \wedge \phi_n(x_1,x_2,\dots, x_n)]
Since ${\cat K}$ is elementary and ${\cat J}^*$ is first order definable relative to ${\cat K}$ it follows, by compactness, that there is some ${\mathbf Y}_{\mathbf A}\in {\cat J}^*_{\mathbf A}$ such that ${\mathbf Y}_{\mathbf A}\models {\mathbf T}_{\mathbf A}$. Now let $f:A\to Y$ be the map that sends each element $a_i$ of ${\mathbf A}$ to its corresponding interpretation in ${\mathbf Y}_{\mathbf A}$. Then, for every atomic formula $\phi(x_1,\dots, x_n)$ in ${\mathcal L}$ and every $n$-tuple $(a_1,\dots, a_n)\in A^n$, we have
\[
{\mathbf A}\models \phi(a_1,\dots, a_n) \Longrightarrow {\mathbf Y}\models \phi(f(a_1),\dots f(a_n)).
\]
Thus, the map $f:A\to Y$ is a homomorphism from ${\mathbf A}$ to ${\mathbf Y}$. Since ${\mathbf Y}\in {\cat J}^*$, there exists a homomorphism $g$ from ${\mathbf Y}$ to some member ${\mathbf J}$ of $\cat J$, and so $g\circ f$ is a homomorphism from ${\mathbf A}$ to ${\mathbf J}$, thus ${\mathbf A}\in {\cat J}^*$.
Hence we have shown that there is a set $\Phi$ of anti-identities such that ${\mathbb H}_{\cat K}^{\gets}({\cat J})={\cat K} \cap \Mod({\Phi})$, namely $\Phi=\Th_{\text{anti}}({\cat J})$.
Now assume there is a single first-order sentence $\phi$ such that ${\mathbb H}_{\cat K}^{\gets}({\cat J})={\cat K} \cap \Mod(\phi)$. Then every model of $\Sigma \cup \Phi$ is also a model of $\Sigma \cup \{\phi\}$. By the Completeness Theorem of first-order logic, we get that %$\Sigma \cup {\mathcal A}\vdash \Sigma \cup \{\phi\}$, which implies
$\Sigma \cup \Phi\vdash \{\phi\}$. It then follows, by compactness, that there is a finite (possibly empty) set of anti-identities $F\subseteq \Phi$ such $\Sigma \cup F\vdash \phi$. Thus, the sentence $\phi$ is logically equivalent to the conjunction of anti-identities $\bigwedge F$ modulo $\Sigma$.
\end{proof}

\section{Failure at the finite level}\label{section:SectionAlg}
We now show that the finite-level version of the Relativised Homomorphism Preservation Theorem~\ref{thm:relhompresthm} fails for algebraic signatures.  If we restrict the class ${\cat K}$ to all finite bounded lattices, then we can find many classes closed under taking homomorphic pre-images that are first-order definable relative to ${\cat K}$, but not definable by universal sentences.
\begin{thm}\label{thm:algfail}
Let ${\cat K}$ be the class of finite bounded lattices in the signature ${\mathcal L}=\{\vee, \wedge, 0, 1\}$ and let ${\mathbf L}$ be any lattice in ${\cat K}$ with $|L|\ge 2$ . The class ${\mathbb H}_{\cat K}^{\gets}({\mathbf L})$ of all finite bounded lattices admitting a homomorphism into ${\mathbf L}$ is strictly first-order definable relative to ${\cat K}$ but not definable by any universal sentence.
\end{thm}

\begin{proof}
The class ${\mathbb H}_{\cat K}^{\gets}({\mathbf L})$ of all finite bounded lattices admitting a homomorphism into ${\mathbf L}$ is definable by a $\forall\exists$ sentence $\phi$ by \cite[Theorem~4.2]{CDJP}. We will show that the class ${\mathbb H}_{\cat K}^{\gets}({\mathbf L})$ is not definable by any universal sentence. Specifically, we show that, for each $n\in {\mathbb N}$, there exists a finite bounded lattice ${\mathbf A}_n$ satisfying the following two conditions.
\begin{enumerate}
\item ${\mathbf A}_n\notin{\mathbb H}_{\cat K}^{\gets}({\mathbf L})$;
\item every $n$-generated sublattice of ${\mathbf A}_n$ belongs to ${\mathbb H}_{\cat K}^{\gets}({\mathbf L})$.
\end{enumerate}
For each $n\in {\mathbb N}$, let ${\mathbf A}_n$ be the lattice depicted in Figure~\ref{figure:A_n}. The lattice ${\mathbf A}_n$ is a ``stack'' of $n$ copies of ${\mathbf M}_{k}$, where $k$ is a natural number strictly greater than $\max\{|L|, 2\}$.

($1$): Let $n\in {\mathbb N}$. The smallest lattice that ${\mathbf A}_n$ collapses to, without identifying $0$ and $1$, is isomorphic to the lattice ${\mathbf M}_{k}$. Since ${\mathbf M}_{k}$ is simple (as $k\ge 3$), every congruence on ${\mathbf M}_{k}$ collapses everything or collapses nothing. Thus if ${\mathbf A}_n$ is to map homomorphically into ${\mathbf L}$, then ${\mathbf L}$ must contain a sublattice isomorphic to ${\mathbf M}_{k}$, but this is not possible since ${\mathbf M}_{k}$ contains more elements than ${\mathbf L}$. Hence ${\mathbf A}_{n}\notin{\mathbb H}_{\cat K}^{\gets}({\mathbf L})$.

($2$): We will show that every $n$-generated sublattice of ${\mathbf A}_n$ has a homomorphism onto the two-element lattice $\mathbbold{2}$, and therefore has a homomorphism into ${\mathbf L}$. Assume that we are selecting $n$ generators from the lattice ${\mathbf A}_{n}$. In order to generate a sublattice of ${\mathbf A}_n$ that does not admit a homomorphism onto $\mathbbold{2}$ (or equivalently, a sublattice that does not contain a prime ideal), we would need to select \emph{at least} three elements from each of the $n$ maximal anti-chains in ${\mathbf A}_n$. Thus, the lattice ${\mathbf P}$ given in Figure~\ref{figure:P} is the smallest sublattice of ${\mathbf A}_n$ that does not admit a homomorphism onto $\mathbbold{2}$. But this lattice is $3n$-generated and not generated by fewer than $3n$ elements. Hence every $n$-generated a sublattice of ${\mathbf A}_n$ admits a homomorphism onto $\mathbbold{2}$, and therefore into ${\mathbf L}$.
\end{proof}

\begin{remark}
Notice that within the class of finite bounded lattices, the (first-order definable) property of admitting a homomorphism onto ${\mathbbold 2}$ is preserved by taking sublattices, and so the above examples also give us a counterexample to the finite-level relativised {\L}os-Tarski Preservation Theorem for algebraic signatures.
\end{remark}
\begin{remark}
Also notice that the proof of Theorem~\ref{thm:algfail} works for the class ${\mathbb H}_{\cat K}^{\gets}({\cat L})$, where ${\cat L}$ is any class of finite bounded lattices containing at least one non-trivial lattice, and for which there exists $k\ge 3$ such that ${\mathbf M}_{k}\notin {\mathbb S}({\cat L})$.
\end{remark}
\begin{remark}
The proof of Theorem \ref{thm:algfail} also works with ${\cat K}$ replaced by any class of finite lattices containing the lattices in Figure \ref{figure:A_n} (for each $n\in\mathbb{N}$).  As the lattice ${\mathbf A}_n$ is modular, the class ${\cat K}$ in Theorem~\ref{thm:algfail} could be taken to be the class of all finite bounded modular lattices.  Alternatively, ${\cat K}$ could to be the finite part of the variety generated by the bounded lattices ${\mathbf A}_n$.  This variety can be shown to coincide with the variety generated by ${\mathbf M}_{\omega}$.
\end{remark}
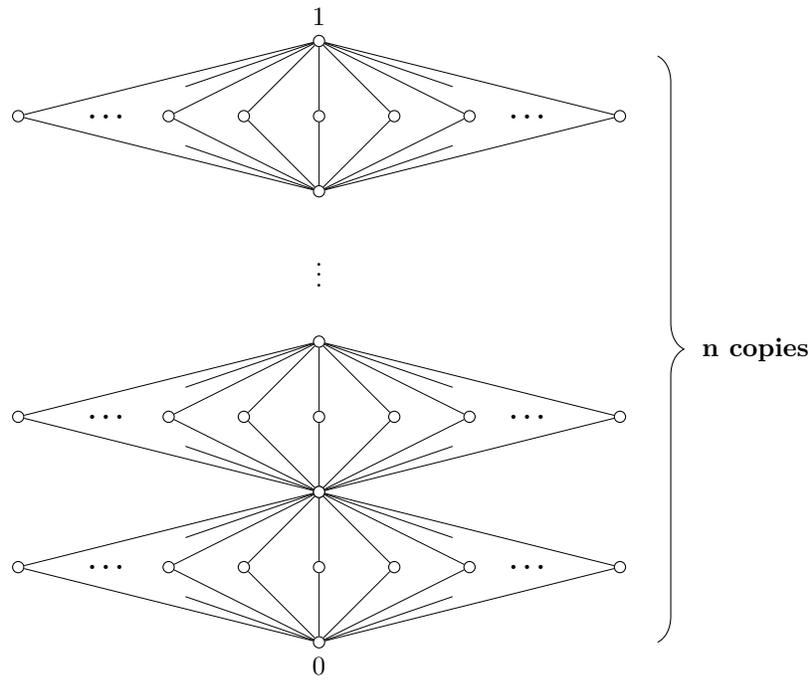
\begin{figure}
\begin{tikzpicture}
%% my styles
 [
 % elements
 unshaded/.style={draw, shape=circle, fill=white, inner sep=1.5pt},
 % lines
 order/.style={very thin}
 ]
%% STACK OF MKAPPA'S
%bottom MKAPPA
   \begin{scope}[xshift=9cm]
       label
      \node at (0,0) {};
      %% elements
      \node[label=below:$0$, unshaded] (bot) at (0,0) {};
      %\node[unshaded] (1) at (-3,1) {};
      \node[unshaded] (2) at (-2,1) {};
      \node[unshaded] (8) at (-4,1) {};
      \node[] at (2.8,1) {$\boldsymbol\dots$};
      \node[] at (-2.8,1) {$\boldsymbol\dots$};
      \node[unshaded] (3) at (2,1) {};
      %\node[unshaded] (4) at (3,1) {};
      \node[unshaded] (9) at (4,1) {};
      \node[unshaded] (5) at (-1,1) {};
      \node[unshaded] (6) at (0,1) {};
      \node[unshaded] (7) at (1,1) {};
      \node[unshaded] (top) at (0,2) {};
      \node[] (half1) at (-1.9,1.35) {};
      \node[] (half2) at (1.9,1.35) {};
      \node[] (half3) at (-1.9,.65) {};
      \node[] (half4) at (1.9,.65) {};
      %% order relation
      %\draw[order] (bot) to (1);
      \draw[order] (bot) to (2);
      \draw[order] (bot) to (3);
      %\draw[order] (bot) to (4);
      \draw[order] (bot) to (5);
      \draw[order] (bot) to (6);
      \draw[order] (bot) to (7);
      \draw[order] (bot) to (8);
      \draw[order] (bot) to (9);
       \draw[order] (bot) to (half3);
      \draw[order] (bot) to (half4);
      %\draw[order] (1) to (top);
      \draw[order] (2) to (top);
      \draw[order] (3) to (top);
      %\draw[order] (4) to (top);
      \draw[order] (5) to (top);
      \draw[order] (6) to (top);
      \draw[order] (7) to (top);
      \draw[order] (8) to (top);
      \draw[order] (9) to (top);
      \draw[order] (half1) to (top);
      \draw[order] (half2) to (top);
      \end{scope}
    %Second MKAPPA
     \begin{scope}[xshift=9cm, yshift=2cm]
      %% label
      %\node at (0,-1);
      %% elements
      \node[unshaded] (bot) at (0,0) {};
      %\node[unshaded] (1) at (-3,1) {};
      \node[unshaded] (2) at (-2,1) {};
      \node[unshaded] (8) at (-4,1) {};
      \node[] at (2.8,1) {$\boldsymbol\dots$};
      \node[] at (-2.8,1) {$\boldsymbol\dots$};
      \node[unshaded] (3) at (2,1) {};
      %\node[unshaded] (4) at (3,1) {};
      \node[unshaded] (9) at (4,1) {};
      \node[unshaded] (5) at (-1,1) {};
      \node[unshaded] (6) at (0,1) {};
      \node[unshaded] (7) at (1,1) {};
      \node[unshaded] (top) at (0,2) {};
      \node[] (half1) at (-1.9,1.35) {};
      \node[] (half2) at (1.9,1.35) {};
      \node[] (half3) at (-1.9,.65) {};
      \node[] (half4) at (1.9,.65) {};
      %% order relation
      %\draw[order] (bot) to (1);
      \draw[order] (bot) to (2);
      \draw[order] (bot) to (3);
      %\draw[order] (bot) to (4);
      \draw[order] (bot) to (5);
      \draw[order] (bot) to (6);
      \draw[order] (bot) to (7);
      \draw[order] (bot) to (8);
      \draw[order] (bot) to (9);
       \draw[order] (bot) to (half3);
      \draw[order] (bot) to (half4);
      %\draw[order] (1) to (top);
      \draw[order] (2) to (top);
      \draw[order] (3) to (top);
      %\draw[order] (4) to (top);
      \draw[order] (5) to (top);
      \draw[order] (6) to (top);
      \draw[order] (7) to (top);
      \draw[order] (8) to (top);
      \draw[order] (9) to (top);
      \draw[order] (half1) to (top);
      \draw[order] (half2) to (top);
    \end{scope}
     %% dots
   \begin{scope}[xshift=9cm]
      \node at (0,5) {$\boldsymbol\vdots$};
    \end{scope}

    %top MKAPPA
    \begin{scope}[xshift=9cm, yshift=6cm]
       %label
      %\node at (0,2) {$1$};
      %% elements
      \node[unshaded] (bot) at (0,0) {};
      %\node[unshaded] (1) at (-3,1) {};
      \node[unshaded] (2) at (-2,1) {};
      \node[unshaded] (8) at (-4,1) {};
      \node[] at (2.8,1) {$\boldsymbol\dots$};
      \node[] at (-2.8,1) {$\boldsymbol\dots$};
      \node[unshaded] (3) at (2,1) {};
      %\node[unshaded] (4) at (3,1) {};
      \node[unshaded] (9) at (4,1) {};
      \node[unshaded] (5) at (-1,1) {};
      \node[unshaded] (6) at (0,1) {};
      \node[unshaded] (7) at (1,1) {};
      \node[label=$1$, unshaded] (top) at (0,2) {};
      \node[] (half1) at (-1.9,1.35) {};
      \node[] (half2) at (1.9,1.35) {};
      \node[] (half3) at (-1.9,.65) {};
      \node[] (half4) at (1.9,.65) {};
      %% order relation
      %\draw[order] (bot) to (1);
      \draw[order] (bot) to (2);
      \draw[order] (bot) to (3);
      %\draw[order] (bot) to (4);
      \draw[order] (bot) to (5);
      \draw[order] (bot) to (6);
      \draw[order] (bot) to (7);
      \draw[order] (bot) to (8);
      \draw[order] (bot) to (9);
       \draw[order] (bot) to (half3);
      \draw[order] (bot) to (half4);
      %\draw[order] (1) to (top);
      \draw[order] (2) to (top);
      \draw[order] (3) to (top);
      %\draw[order] (4) to (top);
      \draw[order] (5) to (top);
      \draw[order] (6) to (top);
      \draw[order] (7) to (top);
      \draw[order] (8) to (top);
      \draw[order] (9) to (top);
      \draw[order] (half1) to (top);
      \draw[order] (half2) to (top);
    \end{scope}
 %curly brace
    \begin{scope}[xshift=13.5cm, yshift=0cm]
    \draw [decorate,decoration={brace,amplitude=10pt}]
    (0,7.8)--(0,0) node[midway, right, xshift=10pt] {\ \textbf{n copies}};
    \end{scope}
\end{tikzpicture}
\caption{The lattice ${\mathbf A}_n$.}\label{figure:A_n}
\end{figure}

\begin{figure}
%%%%%%%%%%%%%%%%%%%%%%%%%%%%%%%%%%%%%
%%%%%%%%%%%%%%%%%%%%%%%%%%%%%%%%%%%%%
\begin{tikzpicture}
%% my styles
 [
 % elements
 unshaded/.style={draw, shape=circle, fill=white, inner sep=1.5pt},
 % lines
 order/.style={very thin}
 ]
%% Stack of M3's
%% bottom M3:
   \begin{scope}[xshift=9cm]
      %% label
      %\node at (0,-1) {The lattice ${\mathbf A}_n$};
      %% elements
      \node[label=below:$0$, unshaded] (bot) at (0,0) {};
      \node[unshaded] (1) at (-1,1) {};
      \node[unshaded] (2) at (0,1) {};
      \node[unshaded] (3) at (1,1) {};
      \node[unshaded] (top) at (0,2) {};
      %% order relation
      \draw[order] (bot) to (1);
      \draw[order] (bot) to (2);
      \draw[order] (bot) to (3);
      \draw[order] (1) to (top);
      \draw[order] (2) to (top);
      \draw[order] (3) to (top);
   \end{scope}
   %%second to bottom M3:
   \begin{scope}[xshift=9cm, yshift=2cm]
      %% label
      %\node at (0,-1) {$\mathbf{M}_3$};
      %% elements
      \node[unshaded] (bot) at (0,0) {};
      \node[unshaded] (1) at (-1,1) {};
      \node[unshaded] (2) at (0,1) {};
      \node[unshaded] (3) at (1,1) {};
      \node[unshaded] (top) at (0,2) {};
      %% order relation
      \draw[order] (bot) to (1);
      \draw[order] (bot) to (2);
      \draw[order] (bot) to (3);
      \draw[order] (1) to (top);
      \draw[order] (2) to (top);
      \draw[order] (3) to (top);
   \end{scope}
   %% dots
   \begin{scope}[xshift=9cm]
      \node at (0,5) {$\boldsymbol\vdots$};
    \end{scope}
   %top M3:
    \begin{scope}[xshift=9cm, yshift=6cm]
      %% label
      %\node at (0,-1) {$\mathbf{M}_3$};
      %% elements
      \node[unshaded] (bot) at (0,0) {};
      \node[unshaded] (1) at (-1,1) {};
      \node[unshaded] (2) at (0,1) {};
      \node[unshaded] (3) at (1,1) {};
      \node[label=$1$, unshaded] (top) at (0,2) {};
      %% order relation
      \draw[order] (bot) to (1);
      \draw[order] (bot) to (2);
      \draw[order] (bot) to (3);
      \draw[order] (1) to (top);
      \draw[order] (2) to (top);
      \draw[order] (3) to (top);
   \end{scope}
   %curly brace
    \begin{scope}[xshift=10.5cm, yshift=0cm]
    \draw [decorate,decoration={brace,amplitude=10pt}]
    (0,7.8)--(0,0) node[midway, right, xshift=10pt] {\ \textbf{n copies}};
\end{scope}

%% dots
   %\begin{scope}[xshift=8cm]
      %\node at (0,1) {$\boldsymbol\dots$};
    %\end{scope}

\end{tikzpicture}
\caption{The sublattice ${\mathbf P}$ of ${\mathbf A}_n$.}\label{figure:P}
\end{figure}
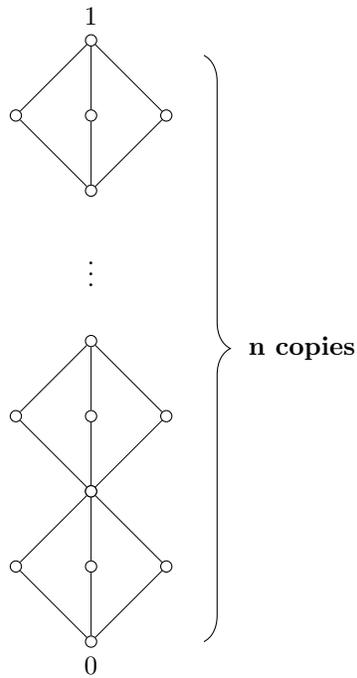
%%%%%%%%%%%%%%%%%%%%%%%%%%%%%%%%%%%%%
%%%%%%%%%%%%%%%%%%%%%%%%%%%%%%%%%%%%%

\section{A Relativised Homomorphism Preservation Theorem} \label{section:SectRelational}
We now turn to the purely relational side of things. In this section, we establish a finite-level version of the Relativised Homomorphism Preservation Theorem for relational signatures.  We present this result as Theorem~\ref{thm:Sausage} below. We first prove the non-relativised version of Theorem~\ref{thm:Sausage}, namely Theorem~\ref{thm:Atserias}, from which Theorem~\ref{thm:Sausage} will follow almost immediately. The proof requires three constructions given in Definition~\ref{defn:nlink}, Definition~\ref{defn:unpinch} and Definition~\ref{defn:npinch} below, borrowed from \cite{LLT}. The quotient structure given in Definition~\ref{defn:npinch}, will play a key role in establishing Theorem~\ref{thm:Sausage}. We follow Jackson and Trotta~\cite{JackTrot} and refer to this quotient structure as the \emph{$n$-pinch} over ${\mathbf A}$. In \cite{LLT}, Larose, Loten and Tardif show that, for a relational structure ${\mathbf A}$, demonstrating membership of the $n$-pinch in $\CSP({\mathbf A})$, for some natural number $n$, is equivalent to $\CSP({\mathbf A})$ admitting the finite duality property. This fact will be crucial in the proof of Theorem~\ref{thm:Sausage}.

We conclude the section by extending Theorem~\ref{thm:Sausage}.  More specifically, we obtain a complete characterisation of not only first-order definable principal colour-families but also first-order definable colour-families in classes ${\cat K}$ closed under forming the $n$-pinch construction and taking disjoint unions; see Theorem~\ref{thm:colourfam}.

In this section, we assume that all structures are finite in a finite purely relational signature ${\mathcal R}=\{R_1, \dots, R_m\}$ with $m\in {\mathbf N}\cup\{0\}$ and $R_i$ of arity $r_i$, for all $i\in\{1,\dots, m\}$, unless otherwise stated.
\begin{defn}\label{defn:nlink}
Let $n,m\in {\mathbb N}\cup\{0\}$, the \emph{$n$-link} of signature ${\mathcal R}=\{R_1, \dots, R_m\}$ is defined to be the structure
\[
{\mathbf L}_n =\langle \{0, 1, \dots, n\};{R_1}^{{\mathbf L}_n}, \dots, {R_m}^{{\mathbf L}_n}\rangle,
\]
where ${R_i}^{{\mathbf L}_n}=\bigcup_{j=1}^{n}{\{j-1, j\}}^{r_i}$, for all $i\in\{1, \dots, m\}$.
\end{defn}

\begin{defn}\label{defn:unpinch}
Let $n\in {\mathbb N}$ and let ${\mathbf A}$ be a structure in the signature ${\mathcal R}=\{R_1, \dots, R_m\}$. Consider the product structure ${\mathbf L}_n\times{\mathbf A}\times {\mathbf A}$. We define an equivalence relation $\sim_n$ on ${\mathbf L}_n\times{\mathbf A}\times {\mathbf A}$ as follows:
\[
(i, a, b) \sim_n (i', a', b') \Longleftrightarrow
\begin{cases} (i, a, b)= (i', a', b'),\quad&\text{or}\\
i=i'=0 \quad\text{and}\quad a=a',\quad&\text{or}\\
i=i'=n \quad\text{and}\quad b=b'.
\end{cases}
\]
\end{defn}
Now using Definition~\ref{defn:nlink} and Definition~\ref{defn:unpinch} we construct the following quotient structure.
\begin{defn}\label{defn:npinch}
Let $n\in {\mathbb N}$ and let ${\mathbf A}$ be a structure in the signature ${\mathcal R}=\{R_1, \dots, R_m\}$, the \textbf{$n$-pinch} over ${\mathbf A}$ is defined to be the structure

\[
{\mathbf P}_n({\mathbf A})=({\mathbf L}_n\times{\mathbf A}\times {\mathbf A})/{\sim_n}.
\]
\end{defn}

Observe that the first coordinate of the elements in ${\mathbf L}_n\times{\mathbf A}\times {\mathbf A}$ is constant on each of the ${\sim_n}$-equivalence classes. In other words, the first coordinate is unchanged by the formation of the quotient structure ${\mathbf P}_n({\mathbf A})$. We therefore use the first coordinate as a means for specifying the position of an element in a copy of ${\mathbf A}$ or ${\mathbf A}\times {\mathbf A}$ in the $n$-pinch: notationally, we define a function $\iota:{\mathbf P}_n({\mathbf A})\to \{0,\dots,n\}$ by $\iota((i,a,b)/{\sim}_{n}):=i$.
We refer to the copy of ${\mathbf A}$ in ${\mathbf P}_n({\mathbf A})$ at position $i=0$ as the \emph{left pinch} of ${\mathbf P}_n({\mathbf A})$, and we refer to the copy of ${\mathbf A}$ at position $i=n$ as the \emph{right pinch} of ${\mathbf P}_n({\mathbf A})$.
We denote the substructures of ${\mathbf P}_n({\mathbf A})$ induced by the sets $B_0=\{(k,a,b)/{\sim_n} |\ k\ne0\}$ and $B_n=\{(k,a,b)/{\sim_n} |\  k\ne n\}$ as ${\mathbf B}_R$ and ${\mathbf B}_L$, respectively. We refer to the copy of ${\mathbf A}\times {\mathbf A}$ in ${\mathbf B}_R$ at position $i=1$ as the \emph{right slice} of ${\mathbf B}_R$, and we refer to the copy of ${\mathbf A}\times {\mathbf A}$ in ${\mathbf B}_L$ at position $i=n-1$ as the \emph{left slice} of ${\mathbf B}_L$.

\begin{figure}
\includegraphics{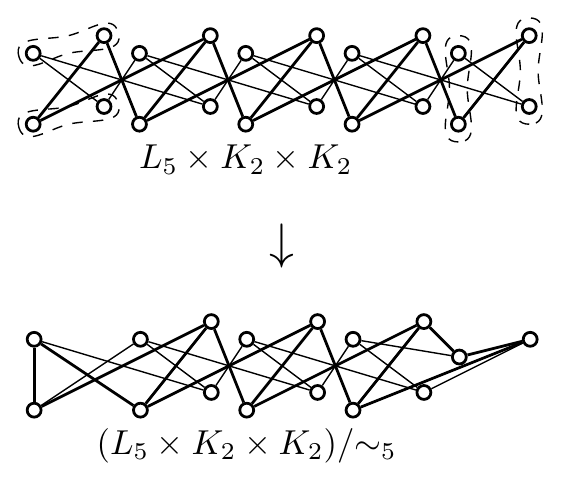}
\caption{The structure ${\mathbf P}_5({\mathbf K}_2)$ as a quotient of $L_5\times K_2\times K_2$.}\label{figure:npinchK2}
\end{figure}
\begin{example}
Let ${\mathbf K}_2$ be the complete graph on two vertices. The formation of ${\mathbf P}_5({\mathbf K}_2)$ is depicted in Figure~\ref{figure:npinchK2}. The graph of ${\mathbf P}_5({\mathbf K}_2)$ can be visualised as a M\"{o}bius strip.
\end{example}

We now collect together some facts that will be used in the proof of our main result.
%The following Lemma is a relativised version of Lemma~$1$ (($2$) $\Longleftrightarrow$ ($6$)) in \cite{JackTrot}, which according to Jackson and %Trotta~\cite{JackTrot} is "at least folklore''. For this reason, we have omitted its proof.
%\begin{lem}\label{lem:RelatvsdFinDual}
%Let $T$ be an elementary theory in a finite relational signature ${\mathfrak R}$ and let ${\cat K}=\Mod_{\textup{fin}}(T)$. Let ${\cat J}$ be a subclass of ${\cat K}$ and consider ${\mathbf H}^{\gets}_{\cat K}({\cat J})$. The following are equivalent\textup:
%\begin{enumerate}[\quad \rm(1)]
%\item ${\mathbf H}^{\gets}_{\cat K}({\cat J})$ is definable by some anti-identity with respect to the class ${\cat K}$\textup;
%\item ${\mathbf H}^{\gets}_{\cat K}({\cat J})$ satisfies the finite duality property with respect to ${\cat K}$.
%\end{enumerate}
%\end{lem}
%in other words, there is a finite set ${\cat S}$ of finite structures in ${\cat K}$ such that a structure ${\mathbf B}\in{\cat K}$ belongs to ${\mathbf H}^{\gets}_{\cat K}({\cat J})$ if and only if ${\mathbf B}$ avoids a homomorphism from every member of ${\cat S}$.

\begin{thm}\label{thm:SausageFacts}
\begin{enumerate}[\quad \rm(1)] Let ${\mathbf A}=\langle A; R_1^{\mathbf A}, \dots, R_m^{\mathbf A}\rangle$ be a finite relational structure.
\item \textup(\cite{LLT}, Lemma~$4.6$\textup) For each natural number $n$, the substructures ${\mathbf B}_R$ and ${\mathbf B}_L$ of ${\mathbf P}_n({\mathbf A})$ admit homomorphisms into ${\mathbf A}$.\label{itm:part1}
\item \textup(\cite{LLT} Theorem~$2.5$ and Theorem~$4.7$\textup)\label{itm:PinchEquiv} The following are equivalent\textup:
\begin{itemize}
\item $\CSP({\mathbf A})$ satisfies the finite duality property\textup;
\item there exists some natural number $n$ such that ${\mathbf P}_n({\mathbf A})$ admits a homomorphism to ${\mathbf A}$.
\end{itemize}
\item \label{itm:pinchsurjhom}\textup(\cite{JackTrot} Lemma~$21$\textup)
For each $n\in {\mathbb N}$, there exists a surjective homomorphism from ${\mathbf P}_{n+1}({\mathbf A})$ to ${\mathbf P}_{n}({\mathbf A})$.
\end{enumerate}
\end{thm}
\begin{remark}
For ($1$), it is routine to check that the map $\phi:B_R\to A$ defined by $\phi((k,a,b)/{\sim_n})=b$ is a homomorphism from ${\mathbf B}_R$ to ${\mathbf A}$ and similarly, that the map $\varphi:B_L\to A$ defined by $\phi((k,a,b)/{\sim_n})=a$ is a homomorphism from ${\mathbf B}_R$ to ${\mathbf A}$, see the proof of Lemma~$4.6$ in \cite{LLT}. The proof of ($2$) depends mostly on the observation that there are only finitely many critical obstructions of a given diameter (see Lemma~$2.4$ of \cite{LLT}. For ($3$), the map defined by \begin{equation*}
\phi_n((i, a, b)/{\sim_n})=
\begin{cases} (i, a, b)/{\sim_n}, \quad&\text{if $i\leq n$}, \\
 (n, a, b)/{\sim_n}, \quad&\text{if $i=n+1$}.
\end{cases}
\end{equation*}
is easily shown to be a surjective homomorphism, see the proof of Lemma~$21$ in \cite{LLT}.
\end{remark}
\begin{defn}
Let ${\cat K}$ be a class of ${\mathcal R}$-structures and let ${\cat F}$ be a subclass of ${\cat K}$. We say that ${\cat K}$ is \emph{closed under forming the $n$-pinch construction over ${\cat F}$}, if for every $n\in{\mathbb N}$ and every ${\mathbf F}\in{\cat F}$, we have that ${\mathbf P}_n({\mathbf F})$ and its substructures ${\mathbf B}_R$ and ${\mathbf B}_L$ belong to the class ${\cat K}$.
\end{defn}
We now state the main result for this section: a finite analogue of the Relativised Homomorphism Theorem~\ref{thm:relhompresthm} in the case where ${\cat J}$ is a single structure.
\begin{thm}\label{thm:Sausage} Let ${\cat K}$ be a class of finite $\mathcal{R}$-structures and let ${\mathbf A}\in {\cat K}$. If ${\cat K}$ is closed under forming the $n$-pinch construction over ${\mathbf A}$ and taking disjoint unions, then ${\mathbb H}_{\cat K}^{\gets}(\mathbf{A})$ is equal to ${\cat K}\cap \Mod(\phi)$ for some first order sentence $\phi$ if and only if there is a finite conjunction of anti-identities $\psi$ such that ${\mathbb H}_{\cat K}^{\gets}(\mathbf{A})={\cat K}\cap \Mod(\psi)$.
\end{thm}
%Let $\Sigma^{+}$ be the set of all non-negated atomic formulae in the signature ${\mathfrak R}$, and let $\phi, \psi \in \Sigma^{+}$. Let $T$ be the theory consisting of sentences true of ${\mathbf A}$ of the following form\textup:
%\begin{itemize}
%\item $\phi$
%\item $\neg\phi$
%\item $\phi \Rightarrow \psi$
%\item $\neg\phi \Rightarrow \neg\psi$.
%\end{itemize}

%Let T be any theory such that $\Mod_{\textup{fin}}(T)$ is closed under the operations ${\mathbb H}$, ${\mathbb S} and ${\mathbb P}_{\textup{fin}}, and disjoint unions.
We first prove a result from which Theorem~\ref{thm:Sausage} will follow almost immediately.
\begin{thm}\label{thm:Atserias}
Let ${\mathbf A}$ be a finite relational structure in the signature ${\mathcal R}$. Then the following are equivalent:
\begin{enumerate}[\quad \rm(1)]
\item $\CSP({\mathbf A})$ satisfies the finite duality property\textup;
\item $\CSP({\mathbf A})$ is definable by a finite conjunction of anti-identities\textup;
\item $\CSP({\mathbf A})$ is first-order definable\textup;
\end{enumerate}
\end{thm}
\begin{proof}
The implication ($1$) $\Rightarrow$ ($2$) is well known (see Lemma $2$ in \cite{JackTrot} for example), and ($2$) $\Rightarrow$ ($3$) is trivial.

We now prove ($3$) $\Rightarrow$ ($1$) by establishing the contrapositive.  We show that for a finite relational structure ${\mathbf A}$, if the class $\CSP({\mathbf A})$ fails the finite duality property, then the property of admitting a homomorphism to ${\mathbf A}$ is not expressible by a first-order sentence. The proof follows an Ehrenfeucht-Fra\"{\i}ss\'e game arguement; we show that, for each non-negative integer $k$, there are two structures ${\mathbf G}_k$ and ${\mathbf H}_k$ such that\textup:
\begin{itemize}
\item Duplicator has a winning strategy in the $k$-round Ehrenfeucht-Fra\"{\i}ss\'e game played on ${\mathbf G}_k$ and ${\mathbf H}_k$\textup, and
\item ${\mathbf H}_k$ belongs to $\CSP({\mathbf A})$, but ${\mathbf G}_k$ does not.
\end{itemize}

If $\CSP({\mathbf A})$ does not satisfy the finite duality property, then Theorem~\ref{thm:SausageFacts} tells us that for each natural number $n$,  there exists no homomorphism from ${\mathbf P}_n({\mathbf A})$ to~${\mathbf A}$, but the substructures ${\mathbf B}_R$ and ${\mathbf B}_L$ of ${\mathbf P}_n({\mathbf A})$ admit homomorphisms to ${\mathbf A}$.  This suggests that we construct our structures ${\mathbf G}_k$ and ${\mathbf H}_k$ using combinations of ${\mathbf P}_n({\mathbf A})$ and its substructures ${\mathbf B}_R$ and ${\mathbf B}_L$, where $n$ is some number defined in terms of $k$.

Let $k$ be a non-negative integer and assume that $n'$ is an integer strictly greater than $2^{k+1}+1$. Now consider ${\mathbf P}_{n'}({\mathbf A})$ and the disjoint union ${\mathbf B}_R \mathbin{\dot\cup} {\mathbf B}_L$.  We build ${\mathbf G}_k$ and ${\mathbf H}_k$ in the following way.  Let ${\mathbf H}_k$ be the structure obtained by taking the disjoint union of ${k+1}$ copies of ${\mathbf B}_R \mathbin{\dot\cup} {\mathbf B}_L$, and let ${\mathbf G}_k$ be ${\mathbf P}_{n'} \mathbin{\dot\cup} {\mathbf H}_k$. See Figure~\ref{figure:sausies} for an abstract depiction of the structures ${\mathbf G}_k$ and ${\mathbf H}_k$.

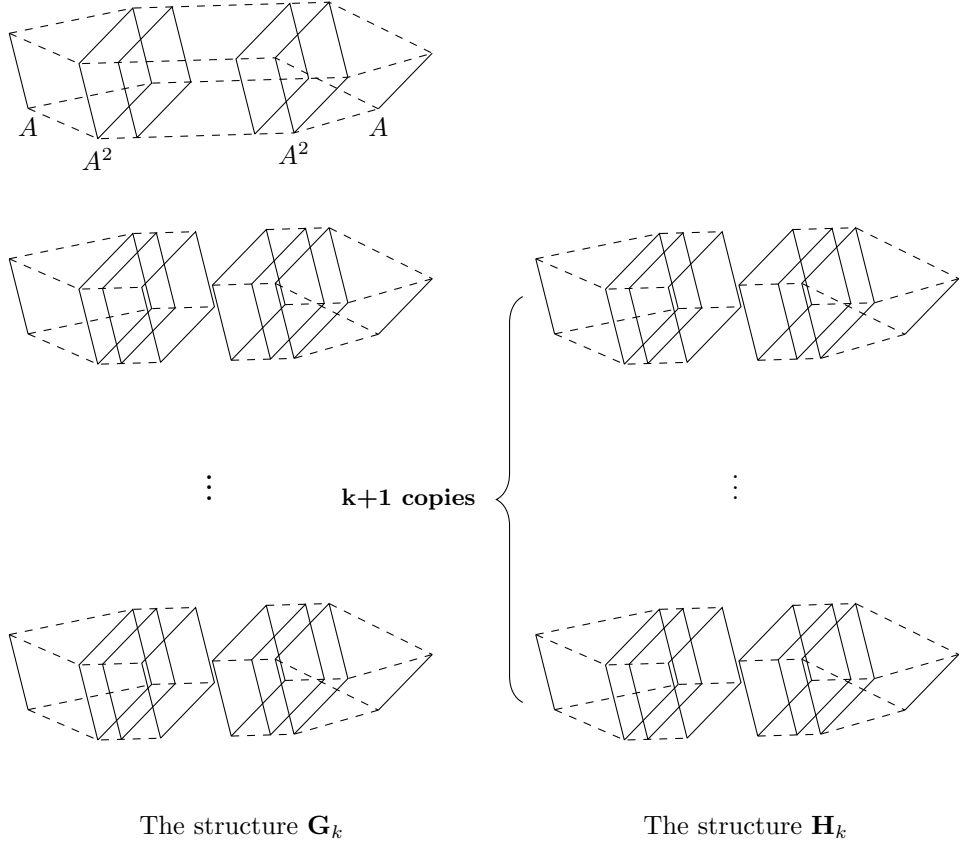
\begin{figure}
\begin{tikzpicture}[x=.5cm,y=0.3cm]
%first sausage with labeling
\begin{scope}[rotate=-15]
\path
node[coordinate,label=below:$A$] (v16) at (-2,0.5) {}
node[coordinate] (v17) at (-3,3.5) {};
\draw (v16)--(v17);
\path
node[coordinate,label=below:$A^2$] (v0) at (0,0) {}
node[coordinate] (v1) at (1,3) {}
node[coordinate] (v2) at (0,6) {}
node[coordinate] (v3) at (-1,3) {};
\draw (v0)--(v1)--(v2)--(v3)--(v0);
\path
node[coordinate] (v4) at (1,0.5) {}
node[coordinate] (v5) at (2,3.5) {}
node[coordinate] (v6) at (1,6.5) {}
node[coordinate] (v7) at (0,3.5) {};
\draw (v4)--(v5)--(v6)--(v7)--(v4);
\path
node[coordinate] (v8) at (4,2) {}
node[coordinate] (v9) at (5,5) {}
node[coordinate] (v10) at (4,8) {}
node[coordinate] (v11) at (3,5) {};
\draw (v8)--(v9)--(v10)--(v11)--(v8);
\path
node[coordinate,label=below:$A^2$] (v12) at (5,2.5) {}
node[coordinate] (v13) at (6,5.5) {}
node[coordinate] (v14) at (5,8.5) {}
node[coordinate] (v15) at (4,5.5) {};
\draw (v12)--(v13)--(v14)--(v15)--(v12);
\path
node[coordinate,label=below:$A$] (v18) at (7,4.5) {}
node[coordinate] (v19) at (8,7.5) {};
\draw (v18)--(v19);
\draw[dashed] (v17)--(v2)--(v6)--(v10)--(v14)--(v19);
\draw[dashed] (v17)--(v3)--(v7)--(v11)--(v15)--(v18);
\draw[dashed] (v16)--(v0)--(v4)--(v8)--(v12)--(v18);
\draw[dashed] (v16)--(v1)--(v5)--(v9)--(v13)--(v19);
\end{scope}
%end of one sausage

%start of next sausage
%\begin{scope}[xshift=6cm,rotate=-15]
%\path node[coordinate] (v16) at (-2,0.5) {} node[coordinate] (v17) at (-3,3.5) {};
%\draw (v16)--(v17);
%\path node[coordinate] (v0) at (0,0) {} node[coordinate] (v1) at (1,3) {}
%node[coordinate] (v2) at (0,6) {} node[coordinate] (v3) at (-1,3) {};
%\draw (v0)--(v1)--(v2)--(v3)--(v0);
%\path node[coordinate] (v4) at (1,0.5) {} node[coordinate] (v5) at (2,3.5) {}
%node[coordinate] (v6) at (1,6.5) {} node[coordinate] (v7) at (0,3.5) {};
%\draw (v4)--(v5)--(v6)--(v7)--(v4);
%\path node[coordinate] (v8) at (4,2) {} node[coordinate] (v9) at (5,5) {}
%node[coordinate] (v10) at (4,8) {} node[coordinate] (v11) at (3,5) {};
%\draw (v8)--(v9)--(v10)--(v11)--(v8);
%\path node[coordinate] (v12) at (5,2.5) {} node[coordinate] (v13) at (6,5.5) {}
%node[coordinate] (v14) at (5,8.5) {} node[coordinate] (v15) at (4,5.5) {};
%\draw (v12)--(v13)--(v14)--(v15)--(v12);
%\path node[coordinate] (v18) at (7,4.5) {} node[coordinate] (v19) at (8,7.5) {};
%\draw (v18)--(v19);
%\draw[dashed] (v17)--(v2)--(v6)--(v10)--(v14)--(v19);
%\draw[dashed] (v17)--(v3)--(v7)--(v11)--(v15)--(v18);
%\draw[dashed] (v16)--(v0)--(v4)--(v8)--(v12)--(v18);
%\draw[dashed] (v16)--(v1)--(v5)--(v9)--(v13)--(v19);
%\end{scope}

%first sausage in second row
\begin{scope}[yshift=-3cm, rotate=-15]
\path
node[coordinate] (v16) at (-2,0.5) {} node[coordinate] (v17) at (-3,3.5) {};
\draw (v16)--(v17);
\path
node[coordinate] (v0) at (0,0) {} node[coordinate] (v1) at (1,3) {}
node[coordinate] (v2) at (0,6) {} node[coordinate] (v3) at (-1,3) {};
\draw (v0)--(v1)--(v2)--(v3)--(v0);
\path
node[coordinate] (v4) at (0.6,0.3) {} node[coordinate] (v5) at (1.6,3.3) {}
node[coordinate] (v6) at (0.6,6.3) {} node[coordinate] (v7) at (-0.4,3.3) {};
\draw (v4)--(v5)--(v6)--(v7)--(v4);
\path
node[coordinate] (v8) at (4.4,2.2) {} node[coordinate] (v9) at (5.4,5.2) {}
node[coordinate] (v10) at (4.4,8.2) {} node[coordinate] (v11) at (3.4,5.2) {};
\draw (v8)--(v9)--(v10)--(v11)--(v8);
\path node[coordinate] (v12) at (5,2.5) {} node[coordinate] (v13) at (6,5.5) {}
node[coordinate] (v14) at (5,8.5) {} node[coordinate] (v15) at (4,5.5) {};
\draw (v12)--(v13)--(v14)--(v15)--(v12);
\path
node[coordinate] (v20) at (3.4,1.7) {} node[coordinate] (v21) at (4.4,4.7) {}
node[coordinate] (v22) at (3.4,7.7) {} node[coordinate] (v23) at (2.4,4.7) {};
\draw (v20)--(v21)--(v22)--(v23)--(v20);
\path
node[coordinate] (v24) at (1.6,0.8) {} node[coordinate] (v25) at (2.6,3.8) {}
node[coordinate] (v26) at (1.6,6.8) {} node[coordinate] (v27) at (0.6,3.8) {};
\draw (v24)--(v25)--(v26)--(v27)--(v24);
\path node[coordinate] (v18) at (7,4.5) {} node[coordinate] (v19) at (8,7.5) {};
\draw (v18)--(v19);
\draw[dashed] (v17)--(v2)--(v6)--(v26);
\draw[dashed] (v22)--(v10)--(v14)--(v19);
\draw[dashed] (v17)--(v3)--(v7)--(v27);
\draw[dashed] (v23)--(v11)--(v15)--(v18);
\draw[dashed] (v16)--(v0)--(v4)--(v24);
\draw[dashed] (v20)--(v8)--(v12)--(v18);
\draw[dashed] (v16)--(v1)--(v5)--(v25);
\draw[dashed] (v21)--(v9)--(v13)--(v19);
\end{scope}
%second sausage in second row
\begin{scope}[xshift=7cm, yshift=-3cm, rotate=-15]
\path
node[coordinate] (v16) at (-2,0.5) {} node[coordinate] (v17) at (-3,3.5) {};
\draw (v16)--(v17);
\path
node[coordinate] (v0) at (0,0) {} node[coordinate] (v1) at (1,3) {}
node[coordinate] (v2) at (0,6) {} node[coordinate] (v3) at (-1,3) {};
\draw (v0)--(v1)--(v2)--(v3)--(v0);
\path
node[coordinate] (v4) at (0.6,0.3) {} node[coordinate] (v5) at (1.6,3.3) {}
node[coordinate] (v6) at (0.6,6.3) {} node[coordinate] (v7) at (-0.4,3.3) {};
\draw (v4)--(v5)--(v6)--(v7)--(v4);
\path
node[coordinate] (v8) at (4.4,2.2) {} node[coordinate] (v9) at (5.4,5.2) {}
node[coordinate] (v10) at (4.4,8.2) {} node[coordinate] (v11) at (3.4,5.2) {};
\draw (v8)--(v9)--(v10)--(v11)--(v8);
\path node[coordinate] (v12) at (5,2.5) {} node[coordinate] (v13) at (6,5.5) {}
node[coordinate] (v14) at (5,8.5) {} node[coordinate] (v15) at (4,5.5) {};
\draw (v12)--(v13)--(v14)--(v15)--(v12);
\path
node[coordinate] (v20) at (3.4,1.7) {} node[coordinate] (v21) at (4.4,4.7) {}
node[coordinate] (v22) at (3.4,7.7) {} node[coordinate] (v23) at (2.4,4.7) {};
\draw (v20)--(v21)--(v22)--(v23)--(v20);
\path
node[coordinate] (v24) at (1.6,0.8) {} node[coordinate] (v25) at (2.6,3.8) {}
node[coordinate] (v26) at (1.6,6.8) {} node[coordinate] (v27) at (0.6,3.8) {};
\draw (v24)--(v25)--(v26)--(v27)--(v24);
\path node[coordinate] (v18) at (7,4.5) {} node[coordinate] (v19) at (8,7.5) {};
\draw (v18)--(v19);
\draw[dashed] (v17)--(v2)--(v6)--(v26);
\draw[dashed] (v22)--(v10)--(v14)--(v19);
\draw[dashed] (v17)--(v3)--(v7)--(v27);
\draw[dashed] (v23)--(v11)--(v15)--(v18);
\draw[dashed] (v16)--(v0)--(v4)--(v24);
\draw[dashed] (v20)--(v8)--(v12)--(v18);
\draw[dashed] (v16)--(v1)--(v5)--(v25);
\draw[dashed] (v21)--(v9)--(v13)--(v19);
\end{scope}
%second sausage in first row
%\begin{scope}[xshift=7cm,yshift=-0.02cm,rotate=-15]
%\path
%node[coordinate] (v16) at (-2,0.5) {} node[coordinate] (v17) at (-3,3.5) {};
%\draw (v16)--(v17);
%\path
%node[coordinate] (v0) at (0,0) {} node[coordinate] (v1) at (1,3) {}
%node[coordinate] (v2) at (0,6) {} node[coordinate] (v3) at (-1,3) {};
%\draw (v0)--(v1)--(v2)--(v3)--(v0);
%\path
%node[coordinate] (v4) at (1,0.5) {} node[coordinate] (v5) at (2,3.5) {}
%node[coordinate] (v6) at (1,6.5) {} node[coordinate] (v7) at (0,3.5) {};
%\draw (v4)--(v5)--(v6)--(v7)--(v4);
%\path
%node[coordinate] (v8) at (4,2) {} node[coordinate] (v9) at (5,5) {}
%node[coordinate] (v10) at (4,8) {} node[coordinate] (v11) at (3,5) {};
%\draw (v8)--(v9)--(v10)--(v11)--(v8);
%\path
%node[coordinate] (v12) at (5,2.5) {} node[coordinate] (v13) at (6,5.5) {}
%node[coordinate] (v14) at (5,8.5) {} node[coordinate] (v15) at (4,5.5) {};
%\draw (v12)--(v13)--(v14)--(v15)--(v12);
%\path
%node[coordinate] (v18) at (7,4.5) {} node[coordinate] (v19) at (8,7.5) {};
%\draw (v18)--(v19);
%\draw[dashed] (v17)--(v2)--(v6)--(v10)--(v14)--(v19);
%\draw[dashed] (v17)--(v3)--(v7)--(v11)--(v15)--(v18);
%\draw[dashed] (v16)--(v0)--(v4)--(v8)--(v12)--(v18);
%\draw[dashed] (v16)--(v1)--(v5)--(v9)--(v13)--(v19);
%\end{scope}

%%first sausage in last row
\begin{scope}[yshift=-6cm, rotate=-15, scale=1]
\node[font=\huge] at (2.1,6) {$\vdots$};
\end{scope}

\begin{scope}[yshift=-8cm, rotate=-15]
\path
node[coordinate] (v16) at (-2,0.5) {} node[coordinate] (v17) at (-3,3.5) {};
\draw (v16)--(v17);
\path
node[coordinate] (v0) at (0,0) {} node[coordinate] (v1) at (1,3) {}
node[coordinate] (v2) at (0,6) {} node[coordinate] (v3) at (-1,3) {};
\draw (v0)--(v1)--(v2)--(v3)--(v0);
\path
node[coordinate] (v4) at (0.6,0.3) {} node[coordinate] (v5) at (1.6,3.3) {}
node[coordinate] (v6) at (0.6,6.3) {} node[coordinate] (v7) at (-0.4,3.3) {};
\draw (v4)--(v5)--(v6)--(v7)--(v4);
\path
node[coordinate] (v8) at (4.4,2.2) {} node[coordinate] (v9) at (5.4,5.2) {}
node[coordinate] (v10) at (4.4,8.2) {} node[coordinate] (v11) at (3.4,5.2) {};
\draw (v8)--(v9)--(v10)--(v11)--(v8);
\path node[coordinate] (v12) at (5,2.5) {} node[coordinate] (v13) at (6,5.5) {}
node[coordinate] (v14) at (5,8.5) {} node[coordinate] (v15) at (4,5.5) {};
\draw (v12)--(v13)--(v14)--(v15)--(v12);
\path
node[coordinate] (v20) at (3.4,1.7) {} node[coordinate] (v21) at (4.4,4.7) {}
node[coordinate] (v22) at (3.4,7.7) {} node[coordinate] (v23) at (2.4,4.7) {};
\draw (v20)--(v21)--(v22)--(v23)--(v20);
\path
node[coordinate] (v24) at (1.6,0.8) {} node[coordinate] (v25) at (2.6,3.8) {}
node[coordinate] (v26) at (1.6,6.8) {} node[coordinate] (v27) at (0.6,3.8) {};
\draw (v24)--(v25)--(v26)--(v27)--(v24);
\path node[coordinate] (v18) at (7,4.5) {} node[coordinate] (v19) at (8,7.5) {};
\draw (v18)--(v19);
\draw[dashed] (v17)--(v2)--(v6)--(v26);
\draw[dashed] (v22)--(v10)--(v14)--(v19);
\draw[dashed] (v17)--(v3)--(v7)--(v27);
\draw[dashed] (v23)--(v11)--(v15)--(v18);
\draw[dashed] (v16)--(v0)--(v4)--(v24);
\draw[dashed] (v20)--(v8)--(v12)--(v18);
\draw[dashed] (v16)--(v1)--(v5)--(v25);
\draw[dashed] (v21)--(v9)--(v13)--(v19);
\end{scope}
%second copy of vdots
\begin{scope}[xshift=7cm, yshift=-6cm, rotate=-15, scale=1]
\node[] at (2.1,6) {$\boldsymbol\vdots$};
\end{scope}
%second copy of sausage in last row
\begin{scope}[xshift=7cm, yshift=-8cm, rotate=-15]
\path
node[coordinate] (v16) at (-2,0.5) {} node[coordinate] (v17) at (-3,3.5) {};
\draw (v16)--(v17);
\path
node[coordinate] (v0) at (0,0) {} node[coordinate] (v1) at (1,3) {}
node[coordinate] (v2) at (0,6) {} node[coordinate] (v3) at (-1,3) {};
\draw (v0)--(v1)--(v2)--(v3)--(v0);
\path
node[coordinate] (v4) at (0.6,0.3) {} node[coordinate] (v5) at (1.6,3.3) {}
node[coordinate] (v6) at (0.6,6.3) {} node[coordinate] (v7) at (-0.4,3.3) {};
\draw (v4)--(v5)--(v6)--(v7)--(v4);
\path
node[coordinate] (v8) at (4.4,2.2) {} node[coordinate] (v9) at (5.4,5.2) {}
node[coordinate] (v10) at (4.4,8.2) {} node[coordinate] (v11) at (3.4,5.2) {};
\draw (v8)--(v9)--(v10)--(v11)--(v8);
\path node[coordinate] (v12) at (5,2.5) {} node[coordinate] (v13) at (6,5.5) {}
node[coordinate] (v14) at (5,8.5) {} node[coordinate] (v15) at (4,5.5) {};
\draw (v12)--(v13)--(v14)--(v15)--(v12);
\path
node[coordinate] (v20) at (3.4,1.7) {} node[coordinate] (v21) at (4.4,4.7) {}
node[coordinate] (v22) at (3.4,7.7) {} node[coordinate] (v23) at (2.4,4.7) {};
\draw (v20)--(v21)--(v22)--(v23)--(v20);
\path
node[coordinate] (v24) at (1.6,0.8) {} node[coordinate] (v25) at (2.6,3.8) {}
node[coordinate] (v26) at (1.6,6.8) {} node[coordinate] (v27) at (0.6,3.8) {};
\draw (v24)--(v25)--(v26)--(v27)--(v24);
\path node[coordinate] (v18) at (7,4.5) {} node[coordinate] (v19) at (8,7.5) {};
\draw (v18)--(v19);
\draw[dashed] (v17)--(v2)--(v6)--(v26);
\draw[dashed] (v22)--(v10)--(v14)--(v19);
\draw[dashed] (v17)--(v3)--(v7)--(v27);
\draw[dashed] (v23)--(v11)--(v15)--(v18);
\draw[dashed] (v16)--(v0)--(v4)--(v24);
\draw[dashed] (v20)--(v8)--(v12)--(v18);
\draw[dashed] (v16)--(v1)--(v5)--(v25);
\draw[dashed] (v21)--(v9)--(v13)--(v19);
\end{scope}

%curly brace
\begin{scope}[xshift=5.65cm, yshift=-7.5cm]
\draw [decorate,decoration={brace,amplitude=10pt}]
(0,0)--(0,18) node[midway, left, xshift=-.5cm] {\small \textbf{k+1 copies}};
\end{scope}

\begin{scope}[yshift=-10.5cm, rotate=-15, scale=1]
\node[] at (3,6) {The structure ${\mathbf G}_k$};
\end{scope}
\begin{scope}[xshift=6.7cm, yshift=-10.5cm, rotate=-15, scale=1]
\node[] at (3,6) {The structure ${\mathbf H}_k$};
\end{scope}
\end{tikzpicture}
\caption{The Ehrenfeucht-Fra\"{\i}ss\'e game ``board".}\label{figure:sausies}
\end{figure}

%%% Local Variables:
%%% mode: latex
%%% TeX-master: t
%%% End

We begin by defining a notion of what it means for two elements in ${\mathbf G}_k$ or ${\mathbf H}_k$ to be ``close to" or ``far away" from each other; but first we need to give a concrete definition of distance between two elements in ${\mathbf G}_k$ or ${\mathbf H}_k$. We will refer to a copy of ${\mathbf P}_{n'}({\mathbf A})$ or ${\mathbf B}_{\text{R}}$ or ${\mathbf B}_{\text{L}}$ in ${\mathbf G}_k$ or ${\mathbf H}_k$ as a \emph{block-component}.  Every element $x$ in ${\mathbf G}_k$ or ${\mathbf H}_k$ is a copy of one of the form $(i,a,b)/{\sim}_{n'}$, thus we may unambiguously extend the function $\iota:(i,a,b)/{\sim}_{n'}\mapsto i$ to  both ${\mathbf G}_k$ and ${\mathbf H}_k$.  Now define the distance between two elements within ${\mathbf G}_k$ by
\[
d_{{\mathbf G}_k}(x,x'):=\begin{cases}
|\iota(x)-\iota(x')|,&\text{ if $x$, $x'$ lie in the same block component,}\\
\infty,&\text{ otherwise,}
\end{cases}
\]
with $d_{{\mathbf H}_k}$ defined identically on ${\mathbf H}_k$.  The subscripts ${\mathbf G}_k$ or ${\mathbf H}_k$ are dropped when no ambiguity arises.
Note that these ``distances'' are not metrics in the strict sense, as distinct points may have distance $0$ (if they lie in the same block component and have the same $\iota$ value).

%Similarly, we define the distance between two elements $x=(i,a,b)/{\sim_n}$ and $x'=(i',a',b')/{\sim_n}$ in ${\mathbf H}_k$ by
%\[
%d_{{\mathbf H}_k}(x,x') = \begin{cases} |i'-i|\quad&\text{if $x$ and $x'$ belong to the same copy of ${\mathbf B}_{\text{R}}$ or ${\mathbf B}_{\text{L}}$},\\
%\infty\quad&\text{otherwise}.
%\end{cases}
%\]

At each round $i\in\{0, 1, \dots, k\}$, we will say that the distance between two elements $x$ and $x'$ in ${\mathbf G}_k$ or ${\mathbf H}_k$ is $\lrge_i$ if it at least $2^{k-i+1}$. Note that this also includes the case where the distance between $x$ and $x'$ is infinite. Assume that Spoiler and Duplicator have played $i'<k$ rounds of the $k$ round game and have selected some elements $g_1, \dots, g_{i'}\in{\mathbf G}_k$ and $h_1, \dots, h_{i'}\in {\mathbf H}_k$. Let $x, x' \in \{g_1, \dots, g_{i'}, h_1, \dots, h_{i'}\}$. If the distance between $x$ and $x'$ is greater than or equal to $\lrge_{i'}$ and there are no previously played elements between them, then $2^{k-{i'}+1}$ is large enough so that Spoiler cannot connect $x$ and $x'$.  More precisely: at each round of play, Spoiler may halve the distance between $x$ and $x'$;  the value of $\lrge_i$ is chosen such that all further plays (to the $k^{\rm th}$ round) will result in a gap of distance at least $\lrge_k=2$ somewhere ``between'' $x$ and $x'$.  See Libkin \cite{Lib} for very similar arguments.

The ``board" on which Spoiler and Duplicator will play the $k$-round Ehrenfeucht-Fra\"{\i}ss\'e game essentially consists of lots of copies of ${\mathbf A}$ and ${\mathbf A} \times {\mathbf A}$ (some of which are connected to one another). Of course, the orientation of ${\mathbf A} \times {\mathbf A}$ is interchangeable, because ${\mathbf A} \times {\mathbf A}$ admits a natural automorphism given by $(a,b)\to (b,a)$.  However it is convenient for our discussion to fix the ordering on these coordinates, so that we make think of  ``pinching" the second coordinate on the left side and the first coordinate on the right side.  In the same way, ${\mathbf B}_{\text{R}}$ is in fact isomorphic to ${\mathbf B}_{\text{L}}$, but it is convenient for us to visualize them with their pinches in dual orientation.

We can use the following assumption to shed some light on Duplicator's strategy to win the $k$-round game. We can assume that Spoiler's first move in the $k$-round Ehrenfeucht-Fra\"{\i}ss\'e game will be in ${\mathbf P}_{n'}({\mathbf A})$ in ${\mathbf G}_k$, since ${\mathbf G}_k\backslash {\mathbf P}_{n'}$ is isomorphic to ${\mathbf H}_k$ and so a first move made by Spoiler outside of ${\mathbf P}_{n'}$ is a wasted one. Assume that Spoiler chooses an element $g_1$ that is distance $l$ strictly less than $\lrge_1$ to the left pinch in ${\mathbf P}_{n'}$ (the case for the right pinch is symmetric). Since $l$ is small enough so that, after round $k$, it is possible for Spoiler to ``connect'' two elements and because the orientations at the left and right pinches differ, Duplicator must choose an element $h_1$ from ${\mathbf H}_k$ that is a copy of the element $g_1$ and in the exact corresponding position, that is, exactly distance $l$ from a left pinch.
Now assume that Spoiler chooses an element $g_1$ that is $\lrge_1$ from both the left and right pinch in ${\mathbf P}_{n'}$. Again, Duplicator is forced to choose an element from a copy of ${\mathbf B}_R$ or ${\mathbf B}_L$ in ${\mathbf H}_k$ and she must select an element $h_1$ that is also distance $\lrge_1$ from any right or left pinch in ${\mathbf H}_k$.  She can achieve this by selecting an exact corresponding element to $g_1$ that is somewhere around the middle of a copy of ${\mathbf B}_R$ or ${\mathbf B}_L$ (her choice of ${\mathbf B}_R$ or ${\mathbf B}_L$ does not matter since they both have diameter
 greater than or equal to $\lrge_i$).

We will show, inductively, that Duplicator can not only maintain partial isomorphism but also preserve the following conditions at each of the rounds $i\in\{0, \dots, k\}$. For $0<l,j<i$,
\begin{enumerate}[\quad \rm(1)]
\item if $d_{{\mathbf G}_k}(g_l, g_j)<2^{k-i+1}$, then $d_{{\mathbf H}_k}(h_l, h_j)=d_{{\mathbf G}_k}(g_l, g_j)$\textup;
\item if $d_{{\mathbf G}_k}(g_l, g_j)\ge2^{k-i+1}$, then $d_{{\mathbf H}_k}(h_l, h_j)\ge2^{k-i+1}$\textup;
\item if $g_j$ is of distance $l<2^{k-i+1}$ from a left pinch, then $h_j$ is of distance $l$ from left pinch\textup;
\item if $g_j$ is of distance $l<2^{k-i+1}$ from a right pinch, then $h_j$ is of distance $l$ from right pinch\textup;
\item if $g_j$ is of distance $l<2^{k-i+1}$ from a left slice, then $h_j$ is of distance $l$ from a left slice\textup;
\item if $g_j$ is of distance $l<2^{k-i+1}$ from a right slice, then $h_j$ is of distance $l$ from a right slice\textup;
\item the dual conditions to the above (in other words, switching the roles of $g$ and $h$).
\end{enumerate}

The base case holds vacuously. For the induction step, suppose that Duplicator has survived $i$ rounds of the game by maintaining the above conditions to the completion of round $i$. Assume that Spoiler is making his $({i+1})$st move in ${\mathbf G}_k$ (the case for ${\mathbf H}_k$ is symmetric). If Spoiler's selection for $g_{i+1}$ is equal to some previously played element $g_l$, where $l<i$, then Duplicator's response should be $h_l$.

Case $1$: Spoiler chooses $g_{i+1}$ so that it has distance greater than or equal to $2^{k-(i+1)+1}$ from any other previously played element in ${\mathbf G}_k$. \\
Case $1$(a): Spoiler chooses $g_{i+1}$ to be distance greater than or equal to $2^{k-(i+1)+1}$ from any pinch or slice in ${\mathbf G}_k$.
Duplicator must select $h_{i+1}$ so that it is also distance greater than or equal to $2^{k-(i+1)+1}$ from any other previously played element and from any pinch or slice in ${\mathbf H}_k$. Duplicator can always find a copy of ${\mathbf B}_{\text{R}}$ or ${\mathbf B}_{\text{L}}$ in ${\mathbf H}_k$ that contains no previously played points: after round $i$ there are at least $2(k+1)-i$ unplayed copies (each of length $2^{k+1}$) remaining in ${\mathbf H}_k$. If Duplicator selects $h_{i+1}$ to be as close as possible to the middle of one of these intervals, then we ensure that the $\lrge_{i+1}$ property and all other conditions of the hypothesis are carried through to the (${i+1}$)st round.

Case $1$(b): Spoiler chooses $g_{i+1}$ so that it is distance $l<2^{k-(i+1)+1}$ from a pinch or slice in ${\mathbf G}_k$.
Without loss of generality, suppose spoiler chooses $g_{i+1}$ so that it is distance $l<2^{k-(i+1)+1}$ from a left pinch in ${\mathbf G}_k$ .
Duplicator must select an element $h_{i+1}$ in ${\mathbf H}_k$ that is exactly distance $l$ from a left pinch in ${\mathbf H}_k$ and such that $h_{i+1}$ is distance greater than or equal to $2^{k-(i+1)+1}$ from any other previously played element. Duplicator can always find a copy of ${\mathbf B}_{\text{L}}$ in  ${\mathbf H}_k$ that contains no previously played elements: after round $i$ there are at least $(k+1)-i$ copies remaining in ${\mathbf H}_k$. Now Duplicator should select an element $h_{i+1}$ in one of these copies corresponding to the element that Spoiler chose in his copy of ${\mathbf B}_{\text{L}}$ (or ${\mathbf P}_{n'}$) relative to the natural embedding of ${\mathbf B}_{\text{L}}$ into ${\mathbf P}_{n'}$, and so that it has distance exactly $l<2^{k-(i+1)+1}$ from the left pinch. Then all the conditions of hypothesis are preserved to the ($i+1$)st round in this case.

Case $2$. Spoiler chooses $g_{i+1}$ so that it has distance strictly less than $2^{k-(i+1)+1}$ from some previously played element $g_l$ ($l<i$) in ${\mathbf G}_k$.

Case $2$(a): Spoiler chooses $g_{i+1}$ to be distance greater than or equal to $2^{k-(i+1)+1}$ from any pinch or slice in ${\mathbf G}_k$.
Without loss of generality assume that Spoiler has selected $g_{i+1}$ to be to the right of $g_l$. There are two sub-cases to consider, however Duplicator's strategy is the same in each case: she should select $h_{i+1}$ in ${\mathbf H}_k$ to be an exact copy of the element $g_{i+1}$, and exactly distance $d$ to the right of $h_l$. We only need to check that conditions $(1)$ and $(2)$ are maintained in this case.
\\
Case $2$(a)(i): If there exists a previously selected element $g_j$ to the right of $g_{i+1}$ that is distance strictly less than $2^{k-(i+1)+1}$, then $d_{{\mathbf G}_k}(g_l, g_j)<2^{k-i+1}$, and so condition ($2$) hypothesis tells us that $d_{{\mathbf H}_k}(h_l, h_j)=d_{{\mathbf G}_k}(g_l, g_j)$. Hence Duplicator's selection of $h_{i+1}$ at exactly distance $d$ to the right of $h_l$ ensures that $d_{{\mathbf H}_k}(h_l, h_{i+1})=d_{{\mathbf G}_k}(g_l, g_{i+1})$ and $d_{{\mathbf H}_k}(h_{i+1}, h_j)=d_{{\mathbf G}_k}(g_{i+1}, g_j)$. Hence all conditions of the hypothesis are maintained in this case.
\\
Case $2$(a)(ii): If the distance of any other previously played element to the right of $g_{i+1}$ is greater than or equal $2^{k-(i+1)+1}$. Then, at round $i$, the distance of any previously played element to the right of $g_l$ was greater than or equal to $2^{k-i+1}$, so condition ($1$) of the hypothesis tells us that the distance of any previously played element to the right of $h_l$ in ${\mathbf H}_k$ is $2^{k-i+1}$. Hence Duplicator's selection of $h_{i+1}$ at exactly distance $d$ to the right of $h_l$ ensures that $d_{{\mathbf H}_k}(h_l, h_{i+1})=d_{{\mathbf G}_k}(g_l, g_{i+1})$ and the distance of any other previously played element to the right of $g_{i+1}$ is greater than or equal $2^{k-(i+1)+1}$. Hence all conditions of the hypothesis are maintained in this case.
\\
Case $2$(b) Spoiler chooses $g_{i+1}$ to be distance strictly less than $2^{k-(i+1)+1}$ from any pinch or slice in ${\mathbf G}_k$.
Without loss of generality assume that Spoiler chooses $g_{i+1}$ to be distance strictly less than $2^{k-(i+1)+1}$ from a left pinch.
\\
Case $2$ (b)(i): If Spoiler chooses $g_{i+1}$ to be to the right of $g_l$ at distance $p$, then at round $i$, the distance $m$ between $g_l$ and the left pinch must have been strictly less than $2^{k-i+1}$. So condition ($3$) of the hypothesis ensures that $h_l$ is also distance $m$ from a left pinch. If Duplicator selects $h_{i+1}$ to be distance $p$ to the right of $h_l$, then $h_{i+1}$ is exactly the same distance ($m+p$) as $g_{i+1}$ from a left pinch, so condition ($3$) of the hypothesis is maintained to this round. It remains to check conditions ($1$) and ($2$): applying the arguments given in Case $2$(a)(i) and Case $2$(a)(ii), respectively guarantees that all the conditions of the hypothesis are maintained in this case.
\\
Case $2$ (b)(ii): If Spoiler chooses $g_{i+1}$ to be to the left of $g_l$ at distance $r$ strictly less than $2^{k-(i+1)+1}$, then $g_l$ must be distance $s$ strictly less than $2^{k-i+1}$ from the left pinch, and so the hypothesis ensures that $h_l$ is also distance $s$ from a left pinch. Duplicator should select $h_{i+1}$ to be distance $r$ to the right of $h_l$, then $h_{i+1}$ is exactly the same distance ($r+s$) as $g_{i+1}$ from a left pinch, and so condition ($3$) of the hypothesis is maintained to this round. This selection also guarantees that condition ($1$) is maintained in this case. If there is no previously played element further to the left of $g_{i+1}$, then we are done. If on the other hand, there exists a previously played element $g_j$ further to the left of $g_{i+1}$, we know that $d_{{\mathbf G}_k}(g_j, g_{i+1})<2^{k-(i+1)+1}$, and so we need to check condition ($1$) of the hypothesis.  At round $i$, we had $d_{{\mathbf G}_k}(g_j, g_l)<2^{k-i+1}$, and so the hypothesis tells us that $d_{{\mathbf H}_k}(h_j, h_l)=d_{{\mathbf G}_k}(g_j, g_l)$, and hence Duplicator's selection ensures that $d_{{\mathbf G}_k}(g_j, g_{i+1})=d_{{\mathbf H}_k}(h_j, h_{i+1})$.
\end{proof}

Our main result now follows almost immediately.

\begin{proof}[\textbf{Proof of Theorem}~\ref{thm:Sausage}]
We prove the contrapositive. Assume that ${\mathbb H}^{\gets}_{\cat K}({\mathbf A})$ is not definable by a finite conjunction of anti-identities modulo $\Th({\cat K})$. Then ${\mathbb H}^{\gets}_{\cat K}({\mathbf A})$ fails the finite duality property relative to the class ${\cat K}$ (clearly, ($1$) $\Rightarrow$ ($2$) of Theorem~\ref{thm:Atserias} relativises to any class). Hence $\CSP({\mathbf A})$ fails the finite duality property and so for every natural number $n$, the structure ${\mathbf P}_n({\mathbf A})$ does not admit a homomorphism to ${\mathbf A}$. Now since the class ${\cat K}$ is closed under forming the $n$-pinch construction over ${\mathbf A}$ and taking disjoint unions, the proof of ($3$) $\Rightarrow$ ($1$) of Theorem~\ref{thm:Atserias} relativises to ${\cat K}$: we have found, for each non-negative integer $k$, two structures ${\mathbf G}_k$ and ${\mathbf H}_k$ in ${\cat K}$ such that\textup:
\begin{enumerate}[\quad \rm(1)]
\item Duplicator has a winning strategy in the $k$-round Ehrenfeucht-Fra\"{\i}ss\'e game played on ${\mathbf G}_k$ and ${\mathbf H}_k$\textup, and
\item ${\mathbf H}_k$ belongs to ${\mathbb H}_{\cat K}^{\gets}({\cat A})$, but ${\mathbf G}_k$ does not.
\end{enumerate}
Hence the property of admitting a homomorphism into ${\mathbf A}$ relative to the class ${\cat K}$ is not definable by a first-order sentence.
\end{proof}

We now complete this section by extending Theorem~\ref{thm:Sausage}.
%More specifically, we obtain a complete description of not only all principal colour-families in some general classes ${\cat K}$, but also all colour-families in these classes (see Proposition~\ref{pro:colourfam}).
The proof follows a similar structure: we first prove a non-relativised version, namely Theorem~\ref{thm:GenrlsdSausageFacts}, from which the extended result will follow almost immediately. The following three results given in Lemma~\ref{lem:Erdos}, Corollary~\ref{cor:ExtendedErdos} and Lemma~\ref{lem:FiniteDuality} are required.  The first lemma appears as Lemma~$2.1$ in \cite{LLT}, and is proved by a probabilistic argument based on that originally given by Erd{\H o}s to show the existence, for each non-negative $k$, of a graph with chromatic number greater than $k$ and no cycles of length less than or equal to $k$. See the proof given in Theorem~$5$ of \cite{FedVar}.

\begin{lem}[\cite{LLT}, Lemma~2.1]\label{lem:Erdos}
Let ${\mathbf A}$ and ${\mathbf B}$ be ${\mathcal R}$-structures. If there exists no homomorphism from ${\mathbf B}$ to ${\mathbf A}$, then for every natural number $n$ there exists an ${\mathcal R}$-structure ${\mathbf C}_n$ of girth at least $n$ such that the following conditions hold\textup:
\begin{enumerate}[\quad \rm(1)]
\item there exists no homomorphism from ${\mathbf C}_n$ to ${\mathbf A}$\textup;
\item there exists a homomorphism from ${\mathbf C}_n$ to ${\mathbf B}$.
\end{enumerate}
\end{lem}

We observe that Lemma~\ref{lem:Erdos} is easily extended to a finite number of relational structures.

\begin{cor}\label{cor:ExtendedErdos}
Let $m\in{\mathbb N}$, let ${\cat F}=\{{\mathbf A}_1, \dots, {\mathbf A}_m\}$ be a finite set of ${\mathcal R}$-structures and let ${\mathbf B}$ be an ${\mathcal R}$-structure. If ${\mathbf B}$ does not admit a homomorphism to ${\mathbf A}_i$, for each $i\in\{1, \dots, m\}$,  then for any natural number $n$, there exists an ${\mathcal R}$-structure ${\mathbf C}_n$ of girth at least $n$ such that the following conditions hold\textup:
\begin{enumerate}[\quad \rm(1)]
\item there exists no homomorphism from ${\mathbf C}_n$ to ${\mathbf A}_i$, for each $i\in\{1, \dots, m\}$\textup;
\item there exists a homomorphism from ${\mathbf C}_n$ to ${\mathbf B}$.
\end{enumerate}
\end{cor}
\begin{proof}
Let $i\in \{1, \dots, m\}$. Since ${\mathbf B}$ does not admit a homomorphism to ${\mathbf A}_i$, we can use Lemma~\ref{lem:Erdos}, to find a structure ${\mathbf C}_n^i$ of girth at least $n$ that admits a homomorphism to ${\mathbf B}$ but does not admit one to ${\mathbf A}_i$. Define ${\mathbf C}_n$ to be the disjoint union of the ${\mathbf C}_n^i$. Clearly, the structure ${\mathbf C}_n$ has girth at least $n$ and does not admit a homomorphism to ${\mathbf A}_i$, for all $i\in\{1, \dots, m\}$.
\end{proof}

%The following lemma is borrowed from \cite{LLT}. It appears as Lemma~$2.4$ there and we do not include its proof.
%\begin{lem}
%For any finite relational signature ${\mathfrak R}$ and any natural number $n$, there exist only a finite number of trees with diameter at most $n$ that admit no proper retractions.
%\end{lem}
The proof of the next lemma is identical, up to replacing Lemma~$2.1$ of \cite{LLT} with its extended version Corollary~\ref{cor:ExtendedErdos}, to the proof of Lemma~$2.4$ in \cite{LLT} . We refer the reader to the proof there.

\begin{lem}\label{lem:FiniteDuality}
Let ${\cat F}$ be a finite set of homomorphism-independent ${\mathcal R}$-structures.
Then the following are equivalent\textup:
\begin{enumerate}[\quad \rm(1)]
\item ${\mathbb H}^{\gets}_{\textup{fin}}({\cat F})$ satisfies the finite duality property\textup;
\item there is an upper bound on the diameter of the critical obstructions for ${\mathbb H}^{\gets}_{\textup{fin}}({\cat F})$.
\end{enumerate}
\end{lem}

Part of the proof of Proposition~$4.1$ in~\cite{LLT} is applicable to a slightly more general setting.  We state this as a lemma after first fixing some useful notation.
\begin{notation}
Let ${\cat F}$ be a class of finite $\mathcal{R}$-structures, let ${\mathbf C}$ be a critical obstruction for the class ${\mathbb H}_\textup{fin}^{\gets}({\cat F})$ and let $x\in C$. We let ${\mathbf C}_x$ denote the substructure of ${\mathbf C}$ induced by the set $C\backslash\{x\}$.
\end{notation}

\begin{lem}\label{lem:Hom2nPinch}
Let $n,m\in{\mathbb N}$, let ${\cat F}=\{{\mathbf A}_1, \dots, {\mathbf A}_m\}$ be a finite set of homomorphism-independent ${\mathcal R}$-structures and let ${\mathbf C}$ be a critical obstruction for the class ${\mathbb H}^{\gets}_{\textup{fin}}({\cat F})$ with diameter at least $n+2$.  Assume that there are $x,y\in C$ and $j\in\{1, \dots, m\}$ with $x$ and $y$ distance $n+2$ apart and with both ${\mathbf C}_x$ and ${\mathbf C}_y$ admitting a homomorphism to ${\mathbf A}_{j}$. Then there exists a homomorphism from ${\mathbf C}$ to ${\mathbf P}_{n}({\mathbf A}_{j})$.
\end{lem}
\begin{proof}  The lemma is essentially borrowed from \cite{LLT}, so we give only a proof sketch.
Let $x$ and $y$ be elements in ${\mathbf C}$ at distance $n+2$ and let $\alpha:{\mathbf C}_{x}\to{\mathbf A}_{j}$ and $\beta:{\mathbf C}_{y}\to{\mathbf A}_{j}$ be homomorphisms.
The maps $\kappa$ and $\phi$ defined below are shown to be homomorphisms in the proof of Proposition~$4.1$ in~\cite{LLT}.
Define $\kappa:C\to L_n$ by \begin{equation*}
\kappa(z) :=
\begin{cases} 0, \quad&\text{if $z=x$}, \\
d_{\mathbf C}(x, z)-1, \quad&\text{if $d_{\mathbf C}(x, z)\leq n+1$ and $z\neq x$},\\
n, \quad&\text{if $d_{\mathbf C}(x, z) \geq n+2$}
\end{cases}
\end{equation*}

%Let $i\in\{1, \dots, m\}$. Let $(z_1, \dots, z_{r_i})\in R_i({\mathbf C}$. If there exists $j\in\{1, \dots, r_i\}$, such that $z_j=x$, then $d_{\mathbf C}(z_j, x)=0$, and so $d_{\mathbf C}(z_p, x)=1$, for all $p\neq j$. So, $(\kappa(z_1), \dots, \kappa(z_{r_i}))=(0, \dots, 0)\in R_i({\mathbf L}_n)$. Now assume that for all $j\in\{1, \dots, r_i\}$, we have that $z_j\neq x$. Then for all $j\in\{1, \dots, r_i\}$, we have  $d_{\mathbf C}(z_j, x)\in\{l, l+1\}$, where $l>0$. If $l< n$, then $\{\kappa(z_1), \dots, \kappa(z_{r_i})\}\subset \{l-1, l\}$. Thus, $(\kappa(z_1), \dots, \kappa(z_{r_i}))\in R_i({\mathbf L}_n)$. If $l\geq n+1$, then $(\kappa(z_1), \dots, \kappa(z_{r_i}))\in R_i({\mathbf L}_n)=(k,\dots, k)\in R_i({\mathbf L}_n)$. Hence $\kappa$ is a homomorphism from ${\mathbf C}$ to ${\mathbf L}_n$.

Now select an element $a\in A$ and define the map $\phi$ from $C$ to $P_n({\mathbf A}_{j})$ by
\begin{equation*}
\phi(z) :=
\begin{cases} ((\kappa(z), \alpha(z), \beta(z))/\sim_n , \quad&\text{if $z \notin\{x, y\}$}, \\
((\kappa(z), \alpha(z), a)/\sim_n , \quad&\text{if $z=x$}, \\
((\kappa(z), a, \beta(z))/\sim_n , \quad&\text{if $z= y$}.
\end{cases}
\end{equation*}
This is the desired homomorphism $\phi:\mathbf{C}\to{\mathbf P}_{n}({\mathbf A}_{j})$.
\end{proof}

We need one more result before we can prove our main result for this section. Recall first that ${\mathbb H}^{\gets}_{\textup{fin}}({\cat F})=\bigcup\limits_{i\in\{1, \dots, m\}}\CSP({\mathbf A}_i)$.
\begin{thm}\label{thm:GenrlsdSausageFacts}
Let $m\in{\mathbb N}$ and let ${\cat F}=\{{\mathbf A}_1, \dots, {\mathbf A}_m\}$ be a finite set of homomorphism-independent ${\mathcal R}$-structures. Then the following are equivalent\textup:
\begin{enumerate}[\quad \rm(1)]
\item ${\mathbb H}^{\gets}_{\textup{fin}}({\cat F})$ is definable by a first-order sentence\textup;
\item For each $i\in\{1, \dots, m\}$, the class $\CSP({\mathbf A}_i)$ satisfies the finite duality property\textup;
\item ${\mathbb H}^{\gets}_{\textup{fin}}({\cat F})$ satisfies the finite duality property\textup;
\item ${\mathbb H}^{\gets}_{\textup{fin}}({\cat F})$ is definable by a finite conjunction of anti-identities.
\end{enumerate}
\end{thm}

\begin{proof}
($1$) $\Rightarrow$ ($2$): We prove the contrapositive. Assume that $\CSP({\mathbf A}_{j})$ fails to satisfy the finite duality property. For each $k\in {\mathbb N}\cup\{0\}$, define $n_k:={2^{k+1}+1}$ and let ${\mathbf P}_{n_k}({\mathbf A}_j)$ be the $(2^{k+1}+1)$-pinch over ${\mathbf A}_{j}$. Let ${\mathbf B}_R$ and ${\mathbf B}_L$ be the substructures of ${\mathbf P}_{n_k}({\mathbf A}_j)$ induced by the sets $B_R=\{(k,a,b)/{\sim_{n'}}|\ k\ne0\}$ and $B_L=\{(k,a,b)/{\sim_{n'}} |\  k\ne n'\}$, respectively. From part (\ref{itm:part1} and \ref{itm:PinchEquiv}) of Theorem~\ref{thm:SausageFacts}, we have that, for all $k\in{\mathbb N}$, the structure ${\mathbf P}_{n_k}({\mathbf A}_{j})$ does not admit a homomorphism to ${\mathbf A}_{j}$, but the substructures ${\mathbf B}_R$ and ${\mathbf B}_L$ both admit homomorphisms to ${\mathbf A}_{j}$. Now consider the disjoint union ${\mathbf P}_{n_k}({\mathbf A}_j)\mathbin{\dot\cup}{\mathbf A}_{j}$. Since ${\cat F}=\{{\mathbf A}_1, \dots, {\mathbf A}_m\}$ is a homomorphism independent set, it follows that, for all $k\in{\mathbb N}\cup\{0\}$, the structure ${\mathbf P}_{n_k}({\mathbf A}_j)\mathbin{\dot\cup} {\mathbf A}_j$ does not belong to the class ${\mathbb H}^{\gets}_{\textup{fin}}({\cat F})$. For each $k\in{\mathbb N}\cup\{0\}$ let ${\mathbf H}_k$ be the structure obtained by taking the disjoint union of ${\mathbf A}_j$ with ${k+1}$ copies of ${\mathbf B}_R \mathbin{\dot\cup} {\mathbf B}_L$, and let ${\mathbf G}_k={\mathbf P}_{n'}({\mathbf A}_j)\mathbin{\dot\cup} {\mathbf H}_k$. Then, for each $k\in{\mathbb N}\cup\{0\}$, we have ${\mathbf H}_k\in \CSP({\mathbf A}_j)\subseteq{\mathbb H}^{\gets}_{\textup{fin}}({\cat F})$, while ${\mathbf G}_k\notin{\mathbb H}^{\gets}_{\textup{fin}}({\cat F})$. Now apply the argument given in Theorem~\ref{thm:Sausage} with the extra condition that whenever Spoiler selects an element $g$ from ${\mathbf G}_k$ in the extra copy of ${\mathbf A}_j$, then Duplicator selects the corresponding element $h$ from the extra copy of ${\mathbf A}_j$ in ${\mathbf H}_k$, and vice-versa. Hence, for each non-negative integer $k$, Duplicator has a winning strategy in the $k$-round Ehrenfeucht-Fra\"{\i}ss\'e game played on ${\mathbf G}_k$ and ${\mathbf H}_k$, and so the property of admitting a homomorphism to ${\cat F}$ is not definable by a first-order sentence.

($2$) $\Rightarrow$ ($3$): We will show that there is a bound on the the diameter of the critical obstructions for ${\mathbb H}^{\gets}_{\textup{fin}}({\cat F})$. The required result will then follow by Lemma~\ref{lem:FiniteDuality}.  Assume that, for each $i\in\{1, \dots, m\}$, the class $\CSP({\mathbf A}_{i})$ satisfies the finite duality property. Then, for each $i\in\{1, \dots, n\}$, there exists a natural number $l_i$ with ${\mathbf P}_{l_i}({\mathbf A}_{i})$ admitting a homomorphism to ${\mathbf A}_{i}$, by part (\ref{itm:PinchEquiv}) of Theorem~\ref{thm:SausageFacts}. Let ${\mathbf C}$ be any critical obstruction for ${\mathbb H}^{\gets}_{\textup{fin}}({\cat F})$. Of course, for any element $z\in C$, the substructure ${\mathbf C}_{z}$ admits a homomorphism to at least one of the ${\mathbf A}_{i}$. Now let $k$ be the diameter of ${\mathbf C}$, and let $x$ and $y$ be elements in $C$ at distance $k$. We need to consider two cases.

\emph{Case $1$}\textup: there exists $j\in\{1, \dots, m\}$ with ${\mathbf C}_{x}$ and ${\mathbf C}_{y}$ admitting homomorphisms to ${\mathbf A}_{j}$. Lemma~\ref{lem:Hom2nPinch} tells us that ${\mathbf C}$ admits a homomorphism to ${\mathbf P}_{k-2}({\mathbf A}_j)$. But since ${\mathbf P}_{l_{j}}({\mathbf A}_{j})$ admits a homomorphism to ${\mathbf A}_{j}$, and therefore ${\mathbf P}_{l}({\mathbf A}_{j})$ admits a homomorphism to ${\mathbf A}_j$, for all $l\geq l_j$ (by part~\ref{itm:pinchsurjhom} of Theorem~\ref{thm:SausageFacts}), we must have that $k$ is strictly less than $l_{j}+2$.

\emph{Case $2$}\textup: there exists $p, q\in\{1, \dots, m\}$ with $p\ne q$ and with ${\mathbf C}_{x}$ admitting a homomorphism to ${\mathbf A}_{p}$ and ${\mathbf C}_{y}$ admittng a homomorphism to ${\mathbf A}_{q}$. We will show that $k$ is less than or equal to $l_1 + l_2 + \dots + l_m + 2m$.
If there exists $x', y'\in C$ and $j\in\{1, \dots, m\}$ with ${\mathbf C}_{x'}$ and ${\mathbf C}_{y'}$ admitting homomorphisms to $A_j$, then from Case $1$, the maximum distance between $x'$ and $y'$ in ${\mathbf C}$ is strictly less than $l_{j}+2$. Thus, since ${\mathbf C}_{x}$ admits a homomorphism to ${\mathbf A}_{p}$, any element $z\in C$ with ${\mathbf C}_{z}$ admitting a homomorphism to ${\mathbf A}_{p}$ must be at most distance $l_p+1$ from $x$. Therefore any element $w\in C$ with distance exactly $l_p+2$ from $x$ must have ${\mathbf C}_{w}$ admitting a homomorphism to ${\mathbf A}_{s}$, where $s\in\{1, \dots,m\}$ and $s\ne p$. Then any element $c\in C$ with distance $l_s+2$ from $w$ (and distance $l_p + l_s+4$ from $x$) must have ${\mathbf C}_{c}$ admitting a homomorphism to ${\mathbf A}_{t}$, for some $t\notin \{p,s\}$. We can continue this process until we run out of ${\mathbf A_i}$'s. It follows that the diameter $k$ of ${\mathbf C}$ is at most $l_1 + l_2 + \dots + l_m +2m$.

The implication ($3$) $\Rightarrow$ ($4$) is routine, see Lemma $2$ in \cite{JackTrot} for example, and the implication ($4$) $\Rightarrow$ ($1$) is trivial.
\end{proof}
 %Since for each $i=1, \dots, n$, there exists $l_i\in {\mathbb N}$ such that ${\mathbf P}_{l_i}({\mathbf A}_{i})$ admits a homomorphism to ${\mathbf A}_{i}$, Lemma~\ref{lem:Hom2nPinch} tells us that the distance between two points $z_1$ and $z_2$ in ${\mathbf C}$ such that ${\mathbf C}_{z_1}$ and ${\mathbf C}_{z_2}$ both admit a homomorphism to ${\mathbf A}_j$, for some $j\in\{1, \dots, n\}$, is at most $l_j-1$.

Now the proof of our extended result follows almost immediately.

\begin{thm}\label{thm:colourfam}
Let $m\in {\mathbb N}$. Let ${\cat K}$ be a class of finite relational structures in the signature ${\mathcal R}$ and let ${\cat F}=\{{\mathbf A}_1, \dots, {\mathbf A}_m\}$ be a finite set of homomorphism-independent structures in ${\cat K}$. If ${\cat K}$ is closed under forming the $n$-pinch construction over ${\cat F}$ and taking disjoint unions, then ${\mathbb H}_{\cat K}^{\gets}(\cat F)$ is equal to ${\cat K}\cap \Mod(\phi)$ for some first order sentence $\phi$ if and only if there is a finite conjunction of anti-identities $\psi$ such that ${\mathbb H}_{\cat K}^{\gets}(\cat F)={\cat K}\cap \Mod(\psi)$.
\end{thm}

\begin{proof}
We prove the contrapositive. Assume that ${\mathbb H}^{\gets}_{\cat K}({\cat F})$ is not definable by a finite conjunction of anti-identities modulo $\Th({\cat K})$. Then ${\mathbb H}^{\gets}_{\cat K}({\cat F})$ fails the finite duality property relative to the class ${\cat K}$, and therefore ${\mathbb H}_\textup{fin}(\cat F)$ fails the finite duality property. So there is some $j\in\{1, \dots, m\}$ with $\CSP({\mathbf A}_{j})$ failing to satisfy the finite duality property, using ($2$) $\Rightarrow$ ($3$) of Theorem~\ref{thm:GenrlsdSausageFacts}.
%We show that, for each non-negative integer $k$, there are two structures ${\mathbf G}_k$ and ${\mathbf H}_k$ in ${\cat K}$ such that\textup:
%\begin{enumerate}[\quad \rm(1)]
%\item Duplicator has a winning strategy in the $k$-round Ehrenfeucht-Fra\"{\i}ss\'e game played on ${\mathbf G}_k$ and ${\mathbf H}_k$\textup, and
%\item ${\mathbf H}_k$ belongs to ${\mathbb H}_{\cat K}^{\gets}({\cat A})$, but ${\mathbf G}_k$ does not.
%\end{enumerate}
Now since the class ${\cat K}$ is closed under forming the $n$-pinch construction over ${\cat F}$ and taking disjoint unions, the proof of ($1$) $\Rightarrow$ ($2$) of Theorem~\ref{thm:GenrlsdSausageFacts} relativises to ${\cat K}$. Hence the property of admitting a homomorphism into ${\cat F}$ relative to the class ${\cat K}$ is not definable by a first-order sentence.
\end{proof}

We have made multiple references to closure under formation of the $n$-pinch and taking disjoint unions.
We conclude this section with some examples of sentences that are closed under these constructions.
\begin{eg}\label{eg:examplelaws}
%Let $R_i$ be an $r_i$-ary relation in the signature ${\mathcal R}=\{R_1, \dots R_m\}$,
Let $\mathcal{R}=\{R_1, \dots, R_m\}$ be a relational signature.  Then a class of $\mathcal{R}$-structures ${\cat F}$ is closed under forming the $n$-pinch \up(for sufficiently large $n$\up) and disjoint unions if it is definable by sentences of any of the following forms\textup: % and $n\in {\mathbb N}\cup\{0\}$.
\begin{enumerate}
\item reflexivity\up: $\forall x\, (x,\dots,x)\in R_i$\textup;
%\item Symmetry: $\forall x_1, \dots, x_{r_i}\ (x_1,\dots,x_{r_i})\in R\rightarrow (x_{1\pi},\dots,x_{{r_i}\pi})\in R$, where $\pi$ is a permutation of $\{1,\dots,{r_i}\}$.
\item any finite system of anti-identities\textup;
\item implications of the  form
\[
\forall x_1, \dots, x_{r_i}\ [(x_1, \dots, x_{r_i})\in R_i\rightarrow (x_{l_1}, \dots, x_{l_{r_j}})\in R_j],
\tag{$*$}\label{eqn:star}\]
where $\{l_1,\dots,l_{r_j}\}\subseteq\{1,\dots,r_i\}$ and $|\{x_1,\dots,x_{r_i}\}|=r_i$.
\end{enumerate}
\end{eg}
\begin{proof}
Closure under disjoint unions is obvious in each case; in case (3) this uses the fact that the conclusion of the implication involves variables only appearing in the premise.  We now discuss  the formation of $n$-pinches (for sufficiently large $n$).

Case (1) is trivial, while case (2) follows from the fact that for $m<n$, the $m$-vertex substructures of the $n$-pinch over $\mathbf{A}$ admit a homomorphism into $\mathbf{A}$ (by Theorem~\ref{thm:SausageFacts}(1)), and hence the $n$-pinch will satisfy any anti-identity in at most $m$ variables that is true on $\mathbf{A}$.

For Case (3), let $(x_1,\dots, x_{r_i})=((k_1, a_1, b_1)/{\sim_n}, \dots ,(k_{r_i}, a_{r_i}, b_{r_i})/{\sim_n})\in R_i^{\mathbf{P}_n(\mathbf{A})}$. We have $\{k_1, \dots, k_{r_i}\}\subseteq \{k, k+1\}$. If $k=0$, then from the definition of $\sim_{n}$, for each $m\in\{1, \dots, r_i\}$, there exists $(k_m, a_m, b_m')\in (k_m, a_m, b_m)/{\sim_n}$ satisfying $((k_1, a_1, b_1'),\dots, (k_{r_i}, a_{r_i}, b_{r_i}'))\in R_i^{\mathbf{L}_n\times \mathbf{A}\times \mathbf{A}}$. It follows that $(a_1, \dots, a_{r_i})\in R_i^{\mathbf A}$ and $(b_1', \dots, b_{r_i}')\in R_i^{\mathbf{A}}$. Note that here we are implicitly using the condition that $|\{x_1,\dots,x_{r_i}\}|=r_i$; if $x_1, \dots, x_{r_i}$ contained repeated elements then we could not guarantee repeated elements in $b_1', \dots, b_{r_i}'$. Now interpreting \eqref{eqn:star} in ${\mathbf A}$, gives $(a_{l_1}, \dots, a_{l_{r_j}})\in R_{j}^{\mathbf A}$ and $(b_{l_1}', \dots, b_{l_{r_j}}')\in R_{j}^{\mathbf A}$, where $\{l_1,\dots, l_{r_j}\}\subseteq \{1, \dots, r_{i}\}$. It follows that $(x_{l_1}, \dots, x_{l_{r_j}})\in R_{j}^{\mathbf{P}_n(\mathbf{A})}$. A similar argument works for $k=n$ and even easier argument works for $k\in \{1, \dots, n-1\}$.
\end{proof}
The class of simple graphs for example is definable by one anti-identity $(\forall x) \neg R(x,x)$ and one implication $(\forall x, y)\ R(x,y) \rightarrow R(y,x)$ of the form described in Example~\ref{eg:examplelaws}(3).

%%%%%%%%%%%%%%%%%%%%%%%%%%%%%%%%%%%%%%%%%%%%%%%%%%%%%%%%%%%%%%%%%%
%%%%%%%%%%%%%%%%%%%%%%%%%%%%%%%%%%%%%%%%%%%%%%%%%%%%%%%%%%%%%%%%%%
\bibliographystyle{plain}

%%%%%%%%%%%%%%%%%%%%%%%%%%%%%%%%%%%%%%%%%%%%%%%%%%%%%%%%%%%%%%%%%%
%%%%%%%%%%%%%%%%%%%%%%%%%%%%%%%%%%%%%%%%%%%%%%%%%%%%%%%%%%%%%%%%%%
\end{document}